\documentclass[a4paper,reqno]{amsart}
\usepackage{amsthm, amssymb}
\usepackage{amsfonts}
\usepackage{amsmath}
\usepackage{mathrsfs}
\usepackage[sans]{dsfont}        
\usepackage[colorlinks=true,citecolor=black,linkcolor=black]{hyperref}
\usepackage{mathtools}
\usepackage{thmtools,cleveref}
\declaretheorem[name=Theorem]{thm}
\declaretheorem[name=Definition, numberlike=thm]{defin}
\declaretheorem[name=Proposition, numberlike=thm]{prop}
\declaretheorem[name=Lemma, numberlike=thm]{lem}
\declaretheorem[name=Corollary, numberlike=thm]{cor}
\declaretheorem[name=Remark, numberlike=thm]{rem}
\declaretheorem[name=Claim]{claim}
\declaretheorem[name=Condition]{cond}

\newcommand\cG{{\mathcal G}}
\newcommand\cH{{\mathcal H}}

\newcommand\cJ{{\mathcal J}}

\newcommand\cM{{\mathcal M}}

\newcommand\cU{{\mathcal U}}

\newcommand\sL{{\mathscr L}}

\newcommand\sC{{\mathscr C}}

\newcommand\sK{{\mathscr K}}

\newcommand\bC{{\mathbb C}}

\newcommand\bN{{\mathbb N}}

\newcommand\bR{{\mathbb R}}
\newcommand\bT{{\mathbb T}}
\newcommand\bZ{{\mathbb Z}}

\newcommand\B{{\mathbf B}}
\renewcommand\L{{\mathbf L}}
\newcommand\ve{\varepsilon}
\newcommand\mixnorm{\sC^\alpha \to \L^1}
\newcommand{\norm}[1]{\left\lVert{#1}\right\rVert}
\newcommand{\abs}[1]{\left\lvert{#1}\right\rvert}
\newcommand\Id{{\mathds{1}}}
\newcommand\ntrans{{n_1}}

\newcommand\ovl{{\Delta}}

\numberwithin{equation}{section}

\author{Peyman Eslami}
\address{Peyman Eslami\\
Dipartimento di Matematica\\
II Universit\`{a} di Roma (Tor Vergata)\\
Via della Ricerca Scientifica, 00133 Roma, Italy.}
\email{{\tt eslami@mat.uniroma2.it}}
\title[Skew products with discontinuities]{Stretched-exponential mixing for $\sC^{1+\alpha}$ skew products with discontinuities}
\thanks{It is my pleasure to thank Carlangelo Liverani and Oliver Butterley for their guidance and support in the preparation of this article. Without them this work would not have been possible. Research supported by INdAM-COFUND Marie Curie Fellowship.}
\date{  \today}

\begin{document}
\begin{abstract}
Consider the skew product $F:\bT^2 \to \bT^2$, $F(x,y)= (f(x),y+\tau(x))$, where $f:\bT^1\to \bT^1$ is a piecewise $\sC^{1+\alpha}$ expanding map on a countable partition and $\tau:\bT^1 \to \bR$ is piecewise $\sC^1$. It is shown that if $\tau$ is not Lipschitz-cohomologous to a piecewise constant function on the joint partition of $\tau$ and $f$, then $F$ is mixing at a stretched-exponential rate.
\end{abstract}

\maketitle
\bibliographystyle{abbrv}

\section{Introduction} \label{intro}
An important problem in the statistical study of chaotic dynamical systems is obtaining a quantitative estimate on the rate of decay of correlations of the system. Such an estimate describes how fast the system looses memory of its past and opens the door to further statistical description of the system. Ideally, one would like to prove an exponential rate of mixing for systems with ``enough'' hyperbolicity. 

The first results on exponential decay of correlations were obtained for uniformly expanding maps and hyperbolic maps (see \cite{Liv95} and references therein). Slower rates of mixing were also obtained for non-uniformly hyperbolic maps \cite{You98,You99,Sar02}.

For flows most of the existing results on exponential decay of correlations pertain to smooth systems or those with a Markov structure (see \cite{Dol98,Liv04,BalVal05,AviGouYoc06} and references therein). For systems with singularities, Chernov  \cite{Che07} obtained a stretched-exponential rate of decay for certain Billiard flows, Baladi and Liverani \cite{BalLiv12} for piecewise cone-hyperbolic contact flows, while Obayashi  \cite{Oba09} obtained exponential decay of correlations for suspension semiflows over piecewise expanding $\sC^2$ maps of the interval using a tower construction and applying the main result of \cite{BalVal05}.

The goal of this article is to introduce a method by which rates of decay of correlations can be obtained for systems with a neutral direction that have discontinuities and are of low regularity (without assuming the existence of a Markov structure). Such systems appear in practice and are of physical relevance. Indeed, the flow of the Lorenz system of ordinary differential equations (see \cite{But14})  and Billiard flows are examples of such systems. Our motivation is to put forward a method to eventually prove exponential mixing rates for these systems; however, in this article we consider the simplest case -- that of a skew product with a neutral direction.  Also, we will prove only a stretched-exponential bound; however, with more delicate estimates one should be able to obtain an exponential bound. We will illustrate the method by considering a 2D skew-product map with an expanding piecewise $\sC^{1+\alpha}$ map in the base and a neutral direction on which the map is a rigid rotation. The method proposed here is a combination of the point of view of standard pairs due to Dolgopyat \cite{Dol04} and further developed by Chernov \cite{CheMar06,CheDol08,CheDol09}, and the oscillatory cancelation mechanism due to Dolgopyat \cite{Dol98}. 

The skew product is introduced in \Cref{setting}. In \Cref{corr} the main theorem is introduced and proven assuming a crucial estimate. The rest of the article is devoted to the proof of this estimate. \Cref{iter} introduces the terminology of standard pairs and standard families. In section \Cref{trans} we introduce the notion of transversality.  \Cref{modi} shows how one can use the  oscillatory cancelation mechanism of Dolgopyat to modify standard families. Finally in \Cref{dens_decay} the main estimate is proven.

\section{The setting} \label{setting} 
Consider the skew-product $F:\bT^{2}\to \bT^{2}$ defined by 
\begin{equation} \label{eq:skewprod}
F:(x,y) \mapsto \left( f(x), y+\tau(x)   \right).
\end{equation}

Assume that 
\begin{equation}\label{eq:C1+}
f:\bT^1 \to \bT^1 \text{ is piecewise } \sC^{1+\alpha}.
\end{equation}
That is, $\bT^1$ can be partitioned into countably many open intervals (modulo a countable set of endpoints of the intervals) such that each open interval is a maximal interval of monotonicity of $f$  and $f$ is $\sC^{1+\alpha}$ on the open interval, extendable to the closed interval. Note that, for every $n \geq 1$, $f^n$ is also piecewise $\sC^{1+\alpha}$.   Having the graph of $f^n$ in mind, we denote by $\cH^n$ the set of inverse branches of $f^n$. We choose to index partition elements of $f^n$ by elements of $\cH^n$. So, we denote the partition of $f^n$ by $\{O_h\}_{h \in \cH^n}$. Note that the domain and range of $h$ are $f^n(O_h)$ and $O_h$, respectively.

Assume that $f$ is expanding. That is, there exist $\lambda$ such that
\begin{equation} \label{eq:expansion}
 \ln 2 < \lambda \text{ and } e^{\lambda n} \leq \abs{(f^n)'} .
\end{equation}

Assume that $f$ satisfies the following distortion bound. There exists a constant $D \geq 0$ such that for every $n \in \bN$
\begin{equation}\label{eq:dist}
\frac{\abs{h'(x)}}{\abs{h'(y)}} \leq e^{D\abs{x-y}^\alpha} \text{, for } h \in \cH^n\text{, and } x,y \in f^n(O_h).
\end{equation}
Note that this condition is implied by a similar condition for the first iterate of $f$, possibly with slightly worse constants. 

Also, assume that, for every $n \in \bN$,
\begin{equation} \label{eq:summability}
\sum_{h \in \cH^n} \sup_{f^n(O_h)} \abs{h'}< \infty.
\end{equation}
This condition is trivially satisfied when $f$ has finitely many branches. This condition is also implied by the similar condition for the first iterate. Indeed, if $\sum_{h \in \cH} \sup_{f(O_h)} \abs{h'} \leq M < \infty$, then $\sum_{h \in \cH^n} \sup_{f^n(O_h)} \abs{h'}\leq M^n <\infty$.

Assume also that $f$ is covering.
\footnote{If $f$ is an expanding piecewise $\sC^2$ map on a finite partition, then covering and mixing are equivalent, see \cite{HofKel82_1,HofKel82_2}.}
That is, 
\begin{equation} \label{eq:covering}
\forall n \in \bN, \exists N(n) \text{ such that } \forall h \in \cH^n, f^{N(n)}\bar O_h = \bT^1,
\end{equation}
where $\bar O_h$ denotes the closure of the open interval $O_h$.

Assume that
\begin{equation}
\tau:\bT^1 \to \bR \text{ is piecewise } \sC^1.
\end{equation}

Assume that $\tau$ is not Lipschitz-cohomologous to a piecewise constant function. That is, 
\begin{equation} \label{eq:ncoh_pwc}\nexists \phi \in Lip (\bT^1, \bR) \text{ such that }  \tau(x) = \phi \circ f(x) - \phi(x) + \psi(x),\end{equation}where $\psi(x)$ is a piecewise constant function on the joint partition of $f$ and $\tau$.

For every $n \in \bN$, define $\tau_n := \sum_{j=0}^{n-1} \tau \circ f^j$. Assume that there exists $C_\tau$ such that for every $n \in \bN$
\begin{equation} \label{eq:roof_regularity}
\abs{(\tau_n \circ h)'} \leq C_\tau \text{, for } h \in \cH^n.
\end{equation}
This condition is easily satisfied if $|\tau'|$ is bounded. Note that in more general cases mentioned earlier, e.g. the case of the Lorenz flow, this condition is not satisfied. For further details see \cite{But14}.

\textbf{Banach space assumptions.}
Let $\norm{\cdot}_{\sC^\alpha}= \abs{\cdot}_\alpha + \norm{\cdot}_{\sC^0}$, where $\abs{\cdot}_\alpha$ is the usual H\"older semi-norm. Suppose there exists a Banach space $\B \subset \L^{1}$ such that the following hold. 
\begin{enumerate}
\item For every $g \in \B$, $\norm{g}_{\L^1} \leq \norm{g}_{\B}$. For every $g$ in $\sC^\alpha$,  $\norm{g}_{\B} \leq \norm{g}_{\sC^{\alpha}}$. 
\item For every $b$, the weighted transfer operator  $\sL_b:\B \to \B$ associated to $f$ (see \eqref{eq:1dim_Lb}) with weight $\xi(x) = e^{ib\tau(x)}$ is bounded, has a spectral radius $\leq 1$ and has essential spectral radius strictly $< 1$.
\item It is possible to approximate $g \in \B$ with $g_\ve \in \sC^\alpha$  such that $\norm{g-g_\ve}_{\L^1} \leq \norm{g}_\B \ve$ and $\norm{g_\ve}_{\sC^\alpha}< \norm{g}_\B \ve^{-(1+\alpha^{-1})}$.
\end{enumerate}

Under the assumptions \eqref{eq:C1+}--\eqref{eq:summability}, it is known that the Banach space of functions of generalized bounded variation  (see \cite[Section 4.2]{But14}) satisfies the above assumptions. 

Suppose that $f$ preserves an absolutely continuous measure $\mu$ with a density $\wp \in \B$ and $(f,\mu)$ is mixing. It can be easily shown that $F$ preserves the absolutely continuous measure $\nu = \mu \times m$, where $m$ is the Lebesgue measure. We will also denote the density of $\nu$ by $\wp$ since it is constant in the vertical direction. The objective of this note is to prove a stretched-exponential decay of correlations for the skew product $F$. Of course, such an estimate implies $(F, \nu)$ is mixing.
\section{Decay of correlations} \label{corr}
For two observables $\phi$ and $\psi$ the correlation coefficients are defined by
\[
\operatorname{cor}_{\phi,\psi}(n) = \int_{\bT^2} \phi \cdot \psi \circ F^n \ d\nu- \int_{\bT^2} \phi \ d\nu\int_{\bT^2} \psi \ d\nu.
\]

Let $\sL:\L^1(\bT^2) \to \L^1(\bT^2)$ be the transfer operator associated to the skew product $F$. That is,
\begin{equation}\label{eq:2dim_L}
\sL g (x,y) =\sum_{z\in F^{-1}(x,y)} g(z)|\det DF^{-1}(z)| = \sum_{w \in f^{-1}(x)}\frac{1}{\abs{f'(w)}}g(w,y-\tau(w)).
\end{equation}

\begin{thm} Suppose the skew product $F$ satisfies assumptions \eqref{eq:skewprod}--\eqref{eq:roof_regularity}. Then, there exist constants $\gamma_3 > 0$  and $C$ such that for every $\phi \in \sC^{\alpha}(\bT^2)$, and $\psi \in \L^\infty(\bT^2)$,
\begin{equation}
\abs{ \operatorname{cor}_{\phi,\psi}(n)} \leq Ce^{-\gamma_3\sqrt{n}}\norm{\phi}_{\sC^\alpha} \norm{\psi}_{\L^\infty} .
\end{equation}
\end{thm}

\begin{proof} 
It suffices to consider $\phi$ with $\int\phi \ d\nu=0$. Hence, it suffices to estimate $\abs{\int_{\bT^2} \phi \cdot \psi \circ F^n \ d\nu} $. Also, if the result holds with $\nu$ replaced by $m$, the 2D Lebesgue measure, then it will hold for $\nu = \wp dm$, by a standard approximation argument (since $\wp \in \B$ and it can be approximated by a $\sC^\alpha$ function using our assumption on the Banach space from \Cref{setting}). Finally, if the result, i.e. stretched-exponential decay, holds for $\phi \in \sC^3(\bT^2 ,\bR)$, then we can show by approximation that it holds for H\"older observables. Note that such approximations will worsen the rate of decay but the rate will remain stretched-exponential. 

Using  \eqref{eq:2dim_L}, we may write
\[
\abs{\int \phi \cdot \psi \circ F^n \ dm} = \abs{\int \sum_{b \in \bZ}\sL_b^n (\hat\phi_{b}) \cdot \hat\psi_{-b} \ dm},
\]
where $\{\hat \phi_b\}_{b\in \bZ}$ are the Fourier coefficients of $\phi$ in the $y$-direction, and 
\begin{equation} \label{eq:1dim_Lb}
 \sL_b^n g = \sum_{h \in \cH^n} e^{ib\tau_n \circ h }\cdot  g \circ h \cdot |h'| \cdot  \Id_{O_h} \circ h .
\end{equation}
Noting that the $\B$-norm is stronger than then $\L^1$-norm,
\begin{equation} \label{eq:b_split}
\begin{split}
\abs{\int \sum_{b \in \bZ}\sL_b^n ( \hat\phi_{b}) \cdot \hat\psi_{-b} \ dm}
 &\leq  \sum_{b\in \bZ} \norm{\sL_b^n ( \hat\phi_b)}_{\L^1} \norm{\hat\psi_{-b}}_{\L^{\infty}} \\
&\leq \sum_{|b|<b_0} \norm{\sL_b^n}_{\B}\norm{ \hat\phi_b}_{\B} \norm{\hat\psi_{-b}}_{\L^{\infty}} 
\\
&+  \sum_{|b| \geq b_0} \norm{\sL_b^n}_{\mixnorm}\norm{ \hat\phi_b}_{\sC^\alpha} \norm{\hat\psi_{-b}}_{\L^{\infty}}.
 \end{split}
\end{equation}

We estimate the second sum above using the following.
\begin{prop} \label{main_estimate}
There exists $\gamma_2 > 0$ and a constant $C$, such that for every $|b| \geq b_0$, for every $n \in \bN$, 
\begin{equation}
\norm{ \sL_b^n}_{\mixnorm } \leq Ce^{-\frac{\gamma_2}{\ln{|b|}} n}.
\end{equation}
\end{prop} This is the main estimate of the article and the rest of the article is devoted to its proof. Assume that this statement holds and let us finish the proof.

Using the regularity of $\phi, \psi$, there exist constants $C$ and $d >1$ (actually $d=2$ works) such that for every $b \in \bZ$,
\begin{equation}\label{eq:Fourier_decay}
\norm{ \hat\phi_b}_{\sC^\alpha} \leq C \norm{\phi}_{\sC^d}|b|^{-d}, \norm{\hat\psi_{-b}}_{\L^\infty} \leq \norm{\psi}_{\L^{\infty}}.
\end{equation}
Using the estimate of \Cref{main_estimate},
\begin{equation}
\sum_{|b| \geq b_0} \norm{\sL_b^n}_{\mixnorm}\norm{ \hat\phi_b}_{\sC^\alpha} \norm{\hat\psi_{-b}}_{\L^{\infty}} \leq\sum_{|b| \geq b_0}  C \norm{\phi}_{\sC^d} \norm{\psi}_{\L^\infty} e^{-\frac{\gamma_2 n}{\ln{|b|}}} |b|^{-d}.\end{equation}

For estimating the sum over $|b|<b_0$ \eqref{eq:b_split}, we use the following result, which is proven in \Cref{prf_small_b}. Note that the constants $C$ and $r$ below depend on $b$ and that is why for large $|b|$ we need a different argument. 
\begin{prop} \label{small_b}
For all $b \neq 0$, there exists $C$ and $r>0$, both depending on $b$, such that for every $n \in \bN$,
\begin{equation}
\norm{\sL^n_b}_{\B} \leq Ce^{-rn}.
\end{equation}
\end{prop}
Using the estimate of \Cref{small_b} for $|b|<b_0$, and the estimate of \Cref{main_estimate} for $|b|>b_0$, it follows that
\[
 \abs{\operatorname{cor}_{\phi,\psi}(n)}  \leq \sum_{|b| \geq b_0}  C_{b_0}C \norm{\phi}_{\sC^d} \norm{\psi}_{\L^\infty} e^{-\frac{\gamma_2 n}{\ln{|b|}}} |b|^{-d}, \text{ for all } n \in \bN.
\]
Estimating the above sum yields a stretched-exponential decay. Indeed, taking $d=2$, one way to estimate $\sum_{|b| \geq b_0}e^{-\frac{\gamma_2 n}{\ln{|b|}}} |b|^{-2}$ is to split the sum into two parts $|b| \leq L$ and $|b| \geq L+1$ to get 
\[
\sum_{|b| \geq b_0}  e^{-\frac{\gamma_2 n}{\ln{|b|}}} |b|^{-2} \leq Le^{-\frac{\gamma_2 n}{\ln{|L|}}} +L^{-1}.
\]
Now choose $L$ so that the two parts of the sum are equal. The solution is $L=e^{\sqrt{\frac{\gamma_2n}{2}}}$, and gives
\[
 Le^{-\frac{\gamma_2 n}{\ln{|L|}}} +L^{-1} = 2L^{-1} \leq 2e^{-\sqrt{\frac{\gamma_2n}{2}}}.
\]
Therefore,
\[
\abs{\operatorname{cor}_{\phi,\psi}(n)} \leq  2 C_{b_0}C \norm{\phi}_{\sC^2} \norm{\psi}_{\L^\infty} e^{-\sqrt{\frac{\gamma_2}{2}}\sqrt{n}}.
\]

\end{proof}

\section{Iteration of Standard Families} \label{iter}
In this section we introduce standard families and their dynamics. We are essentially modelling the evolution of densities under $\sL_b^n$ with the iteration of standard families. 

For $\alpha \in (0,1)$, and a function $\rho:I \to \bC$ define
\begin{equation}
H(\rho) = \sup_{x,y \in I} \frac{\abs{\ln \abs{\rho(x)} - \ln\abs{\rho(y)}}}{\abs{x-y}^\alpha}.
\end{equation}
Let $\arg(\rho) \in [0,2\pi)$ be the argument of $\rho$ written in polar form. All integrals where the measure is not indicated are with respect to the Lebesgue measure. For any measurable set $A$, $|A|$ denotes the Lebesgue measure of $A$.

\begin{defin}[Standard pair] \label{std_pair}
A \emph{standard pair} with associated parameters $a, b, \ve_0$ is a pair $(I, \rho)$ consisting of an open interval $I$ and a function $\rho \in \L^1(I, \bC)$ such that
\begin{equation}\label{eq:width}
|I| < \ve_0 \leq 1; 
\end{equation}
\begin{equation}\label{eq:normalized}
 \int_I \abs{\rho}  =1; 
 \end{equation}
 \begin{equation}\label{eq:mod_regularity}
 H(\rho) \leq a; 
 \end{equation}
 \begin{equation}\label{eq:arg_regularity} 
\abs{\arg(\rho)'} \leq a|b|. 
\end{equation}
\end{defin}

\begin{defin}[Standard family]
A \emph{standard family} $\cG$ is a set of standard pairs  $\{(I_j, \rho_j)\}_{j \in \cJ}$ and an associated measure $w_\cG$ on a countable set $\cJ$.  We require that there exists a constant $B>0$ such that, 
\begin{equation} \label{eq:bd_def}
\abs{\partial_\ve \cG} := \sum_{j \in \cJ} w_\cG(j) \int_{\partial_\ve I_j}\abs{\rho_j}   \leq B\ve,  \text{ for all } \ve<\ve_0,
\end{equation}
where $\partial_\ve I_j$ denotes the $\ve$-boundary of the interval $I_j$.  
If $w_\cG$ is a probability measure, then $\cG$ is called a \emph{standard probability family}. Each standard family induces an absolutely continuous (complex) measure on $\bT^1$ with the density\footnote{By the sum \eqref{eq:family_density} we really mean the sum of the trivially extended standard pairs to densities defined on all of $\bT^1$, that is we set them equal to zero outside their domain.}:
\begin{equation} \label{eq:family_density}
\rho_\cG = \sum_{j \in \cJ} w_\cG(j)  \rho_j.
\end{equation} 
The total weight of a standard family is denoted $\abs{\cG}:= \sum_{j \in \cJ} w_\cG(j)$. The set of standard families with associated parameters $a,b,B, \ve_0$  is denoted $\cM_{a,b,B, \ve_0}$.
\end{defin}

Suppose $a >0$. For positive quantities $A$, $B$, we shall write $A \asymp_a B$ and say that $A$ and $B$ are $a$-comparable if $e^{-a}A \leq B \leq e^aA$. 
Note the following simple facts.
\begin{enumerate}
\item If $A\asymp_{a} B$ and $a'>a$, then $A\asymp_{a'} B$.
\item If $A\asymp_{a} B$, then $B\asymp_{a} A$. 
\item  If $A\asymp_{a} B$ and $A\leq C \leq B$, then $A\asymp_{a} C \asymp_{a} B$. That is, $A$, $C$, and $B$ are pairwise $a$-comparable. It follows that all values between $A$ and $B$ are pairwise $a$-comparable.
\item If $A \asymp_a B$ and $B \asymp_{a'} C$, then $A \asymp_{a+a'} C$. Therefore, $A \asymp_{a+a'} B \asymp_{a+a'} C$.
\item If $A_1 \asymp_a B_1$ and $A_2 \asymp_{a'} B_2$, then $A_1+A_2 \asymp_{\max\{a,a'\}} B_1 + B_2$.
\item If $A_1 \asymp_a B_1$ and $A_2 \asymp_{a'} B_2$, then $A_1A_2 \asymp_{a+a'} B_1 B_2$.
\end{enumerate}

\begin{lem} \label{Fed}
If $(I,\rho)$ satisfies \eqref{eq:mod_regularity}, then for every $J, J' \subset I$ with $\abs{J}\abs{J'}\neq 0$, 
\begin{equation} \label{eq:comp1}
\inf_I \abs{\rho} \asymp_{a}  Avg_J \abs{\rho}  \asymp_{a} Avg_ {J'} \abs{\rho}\asymp_{a} \sup_I\abs{\rho},
\end{equation}
where $Avg_J \abs{\rho} = |J|^{-1} \int_J |\rho| $ is the average of $\abs{\rho}$ on $J$.
\end{lem}
\begin{proof}
Note that \eqref{eq:mod_regularity} implies that for every $x, y \in I$, $e^{-a\abs{x-y}^\alpha}\abs{\rho}(y) \leq \abs{\rho}(x) \leq e^{a\abs{x-y}^\alpha}\abs{\rho}(y)$. This implies $\inf_I \abs{\rho} \asymp_a \sup_I\abs{\rho}$. For the rest, observe that for every $J \subset I$, $\inf_I\abs{\rho} \leq \inf_J \abs{\rho} \leq Avg_J \abs{\rho} \leq \sup_J \abs{\rho} \leq \sup_I\abs{\rho}$; therefore, lying between $a$-comparable quantities, the averages are also $a$-comparable.
\end{proof}

Note that since \eqref{eq:dist} implies $H(h') \leq D$, we may apply \Cref{Fed} to $(f^n(O_h), h')$. It follows that
\begin{equation} \label{eq:int_dist}
\sup_{f^n(O_h)} \abs{h'} \asymp_D Avg_{f^n(O_h)} \abs{h'} = \frac{\abs{O_h}}{\abs{f^n(O_h)}} \text{, for every } n \in \bN, h \in \cH^n.
\end{equation}
For the last equality we have used that $f^n$ is one-to-one on $O_h$.

Given a standard family $\cG \in \cM_{a,b,B,\ve_0}$, we define its $n$-th iterate as follows. 
\begin{defin} [Iteration] \label{iteration}
Let $\cG$ be a standard family with index set $\cJ$ and weight $w_{\cG}$.  
For $(j,h) \in \cJ \times \cH^n$ such that $\abs{f^n(I_j \cap O_h)} \geq \ve_0$, let $\cU_{(j,h)}$ be the index set of a finite partition $\{U_{\ell}\}_{\ell \in \cU_{(j,h)}}$ of the interval $f^n(I_j \cap O_h)$ into open intervals\footnote{Modulo a finite set of endpoints.} of size
\begin{equation} \label{eq:chop_size}
\ve_0/3 \leq \abs{U_\ell} < \ve_0.
\end{equation}
For $(j,h) \in \cJ \times \cH^n$ such that $0<\abs{f^n(I_j \cap O_h)} < \ve_0$ set $\cU_{(j,h)}=\emptyset$. Define 
\begin{equation}
\cJ_{n}:=\{(j,h,\ell)  | (j,h)\in \cJ \times \cH^n, \ell \in \cU_{(j,h)}, I_j\cap O_h \neq \emptyset \}.\footnote{When $\cU_{(j,h)} = \emptyset$, by $(j,h,\ell)$ we mean $(j,h)$.}
\end{equation}  
For every $j_{n}:=(j, h, \ell) \in \cJ_{n}$, define
\begin{eqnarray*}
I_{j_{n}} &:=& f^n(I_{j} \cap O_h) \cap U_\ell, \\
\rho_{j_{n}} &:=& e^{ib\tau_n \circ h}\rho_{j} \circ h |h'| z_{j_{ n}}^{-1}, \text{ where }z_{j_{ n}} :=\int_{I_{j_{ n}}} \abs{\rho_{j}} \circ h \abs{h'}  .
\end{eqnarray*}
Define 
$
\cG_{ n} := \left\{\left(I_{j_{ n}}, \rho_{j_{ n}} \right)\right\}_{j_{ n}\in \cJ_{ n}}
$
and associate to it the measure given by 
\begin{equation}\label{eq:weight_evol}
w_{\cG_{ n}}(j_{n}) = z_{j_{ n}} w_{\cG}(j).
\end{equation}
\end{defin}

\begin{rem} Comparing \eqref{eq:1dim_Lb} with the definition of $\cG_n$ and the measure associated to it \eqref{eq:family_density}, we have
\begin{equation} \label{eq:connection}
\sL^n_b \rho = \rho_{\cG_n}.
\end{equation}
This is the main connection between the evolution of densities under $\sL^n_b$ and the evolution of standard families.
\end{rem}

\subsection{Invariance}
The first thing to show is the invariance of $\cM_{a,b,B,\ve_0}$ under iterations of $\sL_b^n$ for large enough $a,B,n$ and small enough $\ve_0$. 
\begin{rem}\label{fixed_param}
In this section, by $a,B, n$ large and $\ve_0$ small we mean values that satisfy the following inequalities simultaneously. 
\begin{enumerate}
\item $e^{-\lambda \alpha n} + D/a <1$,
\item $e^{-\lambda n} + C_\tau/a <1$,
\item $e^{4a+2D}(2^n e^{-\lambda n}+ e^{-\sigma}) <1$,
\item $\sigma>0$ is such that there exists $\cH_\sigma^n \subset \cH^n$ such that $\cH_0^n:=\cH^n \setminus \cH_\sigma^n$ is finite, and $\sum_{h \in\cH_\sigma^n}\sup\abs{h'} < e^{-\sigma}$,
\item $\ve_0$ is such that for every interval $I$ with $\abs{I}< \ve_0$, $
\#\left\{ h \in \cH^n_0: I \cap O_h \neq \emptyset \right\} \leq 2^n$.
\end{enumerate} The constants $D$, $C_\tau$ were introduced in \Cref{setting}. One may first choose $a$ and $n$ large enough that the first two inequalities hold. Then also choose $\sigma$ (and $n$) large enough that $e^{4a+2D}(2^n e^{-\lambda n}+ e^{-\sigma})<1$. The value of $\ve_0$ is then determined by $n$ and $\sigma$.
\end{rem}

\begin{prop}\label{invariance}
Suppose $\cG \in \cM_{a,b,B,\ve_0}$ is a standard family. For every $n \in \bN$, for every $(I_{j_n},\rho_{j_n}) \in \cG_n$ we have
\begin{equation} \label{eq:normalized_1}
\int_{I_{j_{n}}}\abs{ \rho_{j_{n}} }   =1,
\end{equation}
\begin{equation} \label{eq:mod_reg_1}
H(\rho_{j_{n}}) \leq a (e^{-\lambda \alpha n} + a^{-1}D),
\end{equation}
\begin{equation} \label{eq:arg_reg_1}
\abs{\arg(\rho_{j_{n}})'} \leq a |b|  (e^{-\lambda n} + a^{-1}C_\tau).
\end{equation}
For every $n \in \bN$,
\begin{equation} \label{eq:total_weight_inv}
\abs{\cG_n} = \abs{\cG}.
\end{equation}
For every $a$ and $n$ large, for every $\sigma>0$, if $\ve_0>0$ is small enough, then
\begin{equation} \label{eq:growth_lemma}
 \abs{\partial_\ve \cG_{n}} \leq C_a(2^n+e^{-\sigma} e^{\lambda n})\abs{\partial_{e^{-\lambda n} \ve}\cG}  + C_a \ve_0^{-1}\ve \abs{\cG}, \text{ for all } \ve < \ve_0.
\end{equation}\end{prop}

\begin{proof}
Property \eqref{eq:normalized_1} follows from the definition. 

To show \eqref{eq:mod_reg_1}, note that 
\begin{equation}
H(\rho_{j_{n}})= H \left( h' \cdot (\rho_j \circ h)\right).
\end{equation}
Using the definition of $H(\cdot)$ and noting its properties under multiplication and composition, it follows that
 \[H(\rho_{j_{n}}) \leq H(h') + e^{-\lambda\alpha n}H(\rho_{j}). \]
By  \eqref{eq:dist} we have $H(h') \leq D$, and by assumption $H(\rho_{j})\leq a$, finishing the proof of \eqref{eq:mod_reg_1}.

To show \eqref{eq:arg_reg_1}, note that $\arg(\rho_{j_{n}}) = b \tau_n \circ h + \arg(\rho) \circ h$. Therefore,
\[
\abs{\arg(\rho_{j_{n}})'} \leq |b| \abs{(\tau_n \circ h)'} + \abs{\arg(\rho)'} |h'| \leq |b| C_\tau + a|b|e^{-\lambda n}.
\]

To show \eqref{eq:total_weight_inv}, write
\begin{equation}
\begin{split}
\sum_{j_{n} \in \cJ_{n}} w_{\cG_{n}} (j_n) 
&= \sum_{j_{n} \in \cJ_{n}} w_{\cG}(j) \int_{I_{j_n}} \abs{\rho_{j}} \circ h \abs{h'}  \\
&= \sum_{(j,h) \in \cJ \times \cH^n} \sum_{\ell \in \cU_{(j,h)}} w_{\cG}(j) \int_{f^n(I_j \cap O_h) \cap U_\ell} \abs{\rho_{j}} \circ h \abs{h'}  \\
&=\sum_{(j,h) \in \cJ \times \cH^n} w_{\cG}(j) \int_{f^n(I_{j} \cap O_h)} \abs{\rho_{j}} \circ h \abs{h'}  \\
&=\sum_{j \in \cJ} w_{\cG}(j) \sum_{h \in \cH^n}\int_{f^n(I_{j} \cap O_h)} \abs{\rho_{j}} \circ h \abs{h'}  \\
&= \sum_{j \in \cJ} w_{\cG} (j) \int_{I_j} \abs{\rho_{j}}  = \sum_{j \in \cJ} w_{\cG} (j).
\end{split}
\end{equation}
Note that in the second line we wrote $\cJ \times \cH^n$ instead of $\{(j,h) \in \cJ \times \cH^n | I_j \cap O_h \neq \emptyset\}$. We can do this because if $I_j \cap O_h = \emptyset$, then the corresponding terms are zero.
The third equality follows by summing over all $\ell$ since the intervals $f^n(I_{j} \cap O_h) \cap U_\ell$ form a partition of the interval $f^n(I_{j} \cap O_h)$. The last line is a consequence of change of variables and $\int\abs{\rho_j} $ being equal to $1$.

To prove  \eqref{eq:growth_lemma}, suppose $\sigma>0$. Then, \eqref{eq:summability} implies that there exists $\cH^n_\sigma \subset \cH^n$ such that
\begin{equation}
\cH^n_0  :=\cH^n \setminus \cH^n_\sigma \text{ is finite},
\end{equation}
and
\begin{equation} \label{eq:tale_sigma}
\sum_{h \in \cH^n_\sigma}\sup_{f^n(O_h)}\abs{h'} < e^{-\sigma}.
\end{equation}
Since $\cH^n_0$ is finite, choose $\ve_0 = \ve_0(\cH^n_0)$ such that\footnote{There is some freedom here to choose $\ve_0$. The value of $\ve_0$ depends on the partition $\{O_h\}_{h \in \cH^n}$ and the value of $\sigma$. Note that since we only use \eqref{eq:growth_lemma} with a fixed $n$, the value of $\ve_0$ causes no problems even if it is very small. The optimal value depends on the underlying system.} for every interval $I$ with $\abs{I}< \ve_0$,
\begin{equation}
\#\left\{ h \in \cH^n_0: I \cap O_h \neq \emptyset \right\} \leq 2^n.
\end{equation}

Suppose $\ve < \ve_0$. We have, by definition, 
\[
\abs{\partial_\ve \cG_{n}} := \sum_{j_{n} \in \cJ_{n}} w_{\cG_{n}}(j_n)\int_{\partial_\ve I_{j_{n}}} \abs{\rho_{j_{n}}}.
\]
We split the sum into two parts according to whether $\cU_{(j,h)} \neq \emptyset$ or $\cU_{(j,h)} = \emptyset$. The two parts are respectively,
\[
\sum_{\{j_{n} \in \cJ_{n}| \cU_{(j,h) \neq \emptyset}\}} w_{\cG_{n}}(j_n)\int_{\partial_\ve I_{j_{n}}} \abs{\rho_{j_{n}}} = \sum_{j \in \cJ}\sum_{\{h \in \cH^n| f^n(I_j \cap O_h) \geq \ve_0\}} \sum_{\ell \in \cU_{(j,h)}} w_{\cG_{n}}(j_n)\int_{\partial_\ve I_{j_{n}}} \abs{\rho_{j_{n}}},
\]
and
\[
\sum_{\{j_{n} \in \cJ_{n}| \cU_{(j,h) = \emptyset}\}} w_{\cG_{n}}(j_n)\int_{\partial_\ve I_{j_{n}}} \abs{\rho_{j_{n}}} = \sum_{j \in \cJ}\sum_{\{h \in \cH^n| f^n(I_j \cap O_h) < \ve_0\}} w_{\cG_{n}}(j_n)\int_{\partial_\ve I_{j_{n}}} \abs{\rho_{j_{n}}}.
\]

\textbf{The case $\cU_{(j,h)} \neq \emptyset$:} First note that \eqref{eq:mod_reg_1} implies $H(\rho_{j_n}) \leq a$ for sufficiently large $n$. Therefore, \eqref{eq:comp1} implies 
$\abs{\partial_\ve I_{j_n}}^{-1} \int_{\partial_\ve I_{j_n}}  \abs{\rho_{j_n}}  \asymp_a \abs{ I_{j_n}}^{-1} \int_{I_{j_n}}\abs{\rho_{j_n}} = \abs{ I_{j_n}}^{-1}$. 
Hence,
\begin{equation}\label{eq:bd_int}
\int_{\partial_\ve I_{j_n}} \abs{\rho_{j_n}}\asymp_a  \frac{\abs{\partial_\ve I_{j_n}}}{\abs{ I_{j_n}}} =  \frac{\abs{\partial_\ve I_{j_n}}}{\abs{f^n(I_j\cap O_h)\cap U_\ell}}. 
\end{equation}
Observe that by definition, $w_{\cG_{n}}(j_n) = w_{\cG} (j) z_{j_n}= w_{\cG} (j) \int_{f^n(I_j \cap O_h) \cap U_\ell} \abs{\rho_j} \circ h \abs{h'} $. Hence by change of variables $w_{\cG_{n}}(j_n) =w_{\cG} (j) \int_{I_j \cap O_h \cap f^{-n}(U_\ell)} \abs{\rho_j}$. Since $H(\rho_j)\leq a$ and $\int_{I_j} \abs{\rho_j}=1$, it follows by \eqref{eq:comp1} that
\begin{equation}\label{eq:bd_w}
w_{\cG_{n}}(j_n)  \asymp_a w_{\cG}(j) \frac{\abs{I_j \cap O_h \cap f^{-n}(U_\ell)}}{\abs{I_j}}.
\end{equation}
Putting \eqref{eq:bd_int} and \eqref{eq:bd_w} together, and then using the distortion estimate \eqref{eq:int_dist},
\begin{equation} \label{eq:pre_sum_l}
\begin{split}
w_{\cG_{n}}(j_n) \int_{\partial_\ve I_{j_n}} \abs{\rho_{j_n}} &\asymp_{2a} w_{\cG}(j) \frac{\abs{\partial_\ve I_{j_n}}\abs{I_j \cap O_h \cap f^{-n}(U_\ell)}}{\abs{I_j}\abs{f^n(I_j\cap O_h)\cap U_\ell}} \\
&\leq_{2a+D} w_{\cG}(j) \abs{\partial_\ve I_{j_n}}  \abs{I_j}^{-1} \sup_{f^n(I_j \cap O_h)} \abs{h'}.\\
&\leq_{2a+D} w_{\cG}(j) 2\ve \abs{I_j}^{-1} \sup_{f^n(I_j \cap O_h)} \abs{h'}.
\end{split}
\end{equation}
Note that by the notation $\leq_a B$ we mean $\leq e^aB$. To finish the estimate, we need to sum over $\ell$, then sum over $h$ such that $f^n(I_j \cap O_h) \geq \ve_0$ and then over $j \in \cJ$. Note that since each interval $f^n(I_j \cap O_h)$ is chopped into intervals of size $\geq \ve_0/3$, the number of elements in $\cU_{(j,h)}$ is bounded by $3\ve_0^{-1}\abs{f^n(I_j \cap O_h)}$. Therefore, summing \eqref{eq:pre_sum_l} over all $\ell \in \cU_{(j,h)}$, and then using the distortion estimate \eqref{eq:int_dist} yields, 
\[
2\ve 3\ve_0^{-1} w_{\cG}(j) \abs{I_j}^{-1} \sup_{f^n(I_j \cap O_h)} \abs{h'} \abs{f^n(I_j \cap O_h)} \leq_{2a+2D} 6\ve \ve_0^{-1} w_{\cG}(j) \abs{I_j}^{-1}\abs{I_j \cap O_h}.
\]
Summing over $h$ such that $f^n(I_j\cap O_h) \geq \ve_0$  and noting that this is no greater than summing over all $h \in \cH^n$ yields 
\[
\leq_{2a+2D} 6\ve \ve_0^{-1} w_{\cG}(j) \abs{I_j}^{-1}\abs{I_j}= 6\ve \ve_0^{-1}w_\cG(j).
\]
Finally, summing over $j \in \cJ$ yields,
\begin{equation}\label{eq:bd_chopped_pairs}
\sum_{\{j_n \in \cJ_n| \cU_{(j,h)} \neq \emptyset\}} w_{\cG_{n}}(j_n)\int_{\partial_\ve I_{j_{n}}} \abs{\rho_{j_{n}}} \leq_{2(a+D)}6\ve \ve_0^{-1}\abs{\cG}.
\end{equation}

\textbf{The case $\cU_{(j,h)} = \emptyset$:} We need to estimate: 
\[
\sum_{\{j_{n} \in \cJ_{n}| \cU_{(j,h) = \emptyset}\}} w_{\cG_{n}}(j_n)\int_{\partial_\ve I_{j_{n}}} \abs{\rho_{j_{n}}} = \sum_{j \in \cJ}\sum_{\{h \in \cH^n| f^n(I_j \cap O_h) < \ve_0\}} w_{\cG_{n}}(j_n)\int_{\partial_\ve I_{j_{n}}} \abs{\rho_{j_{n}}}.
\]
We further split the second sum into two parts, one over $\cH^n_0$ and the other over $\cH^n_\sigma$.

For the sum over $\cH^n_0$, we have the bound
\[
\begin{split}
& \sum_{j \in \cJ} w_{\cG}(j) \sum_{\{h \in \cH_0^n| f^n(I_j \cap O_h) < \ve_0\}}\int_{\partial_\ve f^n (I_j \cap O_h)} \abs{\rho_{j}} \circ h \abs{h'}   \\ 
&\leq
\sum_{j \in \cJ} w_{\cG}(j) \sum_{\{h \in \cH_0^n| f^n(I_j \cap O_h) < \ve_0\}} \left[ \int_{f^n\left(\partial_{e^{-\lambda n}\ve} I_{j} \cap O_h\right)} + \int_{f^n\left(\partial_{e^{-\lambda n}\ve} O_h\cap I_{j}\right)} \right] \abs{\rho_{j}} \circ h \abs{h'}     \\
&\leq \abs{\partial_{e^{-\lambda n} \ve} \cG} + \sum_{j \in \cJ} w_{\cG}(j) \sum_{\{h \in \cH_0^n| f^n(I_j \cap O_h) < \ve_0\}}  \int_{\partial_{e^{-\lambda n}\ve} O_h \cap I_{j}} \abs{\rho_{j}}  . \\
\end{split}
\]
The first inequality holds because if a point is at a distance less than $\ve$ from the boundary of $f^n (I_{j} \cap O_h)$, then its preimage must be at a distance $e^{-\lambda n}\ve$ from the boundary of $I_{j}$ or from the boundary of $O_h$.  The second inequality is a consequence of change of variables and $f^n$ being one-to-one on $I_j\cap O_h$. 

Since $\abs{I_j} < \ve_0$, by the choice of $\ve_0$, it follows that $I_{j}$ intersects at most $2^n$ intervals $O_h$. Also, \eqref{eq:comp1} implies that $Avg_{\partial_{e^{-\lambda n}\ve} O_h \cap I_{j}}\abs{\rho_j} \asymp_a Avg_{\partial_{e^{-\lambda n}\ve} I_{j}} \abs{\rho_j}$. This in turn implies $\int_{\partial_{e^{-\lambda n}\ve} O_h \cap I_{j}}\abs{\rho_j} \leq_a \int_{\partial_{e^{-\lambda n}\ve} I_j }\abs{\rho_j}$ because $\abs{\partial_{e^{-\lambda n}\ve} O_h \cap I_{j}} \leq \abs{\partial_{e^{-\lambda n}\ve} I_j }$. Therefore,
\begin{equation} \label{eq:bd_short_pairs}
\begin{split}
& \sum_{j \in \cJ} w_{\cG}(j) \sum_{\{h \in \cH_0^n| f^n(I_j \cap O_h) < \ve_0\}}\int_{\partial_\ve f^n (I_j \cap O_h)} \abs{\rho_{j}} \circ h \abs{h'}  \\ &\leq_a \abs{\partial_{e^{-\lambda n} \ve} \cG} +  2^n\sum_{j \in \cJ} w_{\cG_{k}}(j)\int_{\partial_{e^{-\lambda n}\ve} I_{j}} \abs{\rho_{j}}  \\
&\leq_a 2^n 2\abs{\partial_{e^{-\lambda n} \ve} \cG}.
\end{split}
\end{equation}

For the sum over $\cH^n_\sigma$, similarly to \eqref{eq:pre_sum_l}, we have the bound
\[
\sum_{\{j_n | \cU_{(j,h)} = \emptyset\}} w_{\cG_{n}}(j_n)\int_{\partial_\ve I_{j_{n}}} \abs{\rho_{j_{n}}} \leq_{2a+D} \sum_{j \in \cJ} w_{\cG}(j)  \sum_{\{h \in \cH_\sigma^n| f^n(I_j \cap O_h) < \ve_0\}}\abs{I_j}^{-1}\abs{\partial_\ve I_{j_n}}\sup_{f^n(I_j \cap O_h)} \abs{h'}.
\]
Multiplying and dividing the right hand side by $\abs{\partial_\ve I_j}$ and using $\abs{\partial_\ve I_{j_n}}\abs{\partial_\ve I_j}^{-1} \leq 1$, the right hand side is 
\[ \leq_{2a+D} \sum_{j \in \cJ} w_{\cG}(j) \abs{\partial_\ve I_{j_n}} \abs{I_j}^{-1} \sum_{\{h \in \cH_\sigma^n| f^n(I_j \cap O_h) < \ve_0\}}\sup_{f^n(I_j \cap O_h)} \abs{h'}.
\]
Notice that $\abs{\partial_\ve I_{j_n}} \abs{I_j}^{-1} \asymp_a \int_{\partial_\ve I_j} \abs{\rho_j}$. Therefore, the above quantity is
\[ 
\leq_{3a+D} \sum_{j \in \cJ} w_{\cG}(j)\int_{\partial_\ve I_j} \abs{\rho_j} \sum_{\{h \in \cH_\sigma^n| f^n(I_j \cap O_h) < \ve_0\}}\sup_{f^n(I_j \cap O_h)} \abs{h'}.
\]

Using \eqref{eq:tale_sigma}, the estimate for the sum over $\cH^n_\sigma$ is
\begin{equation} \label{eq:bd_tail}
\leq_{3a+D} e^{-\sigma}\sum_{j \in \cJ} w_{\cG}(j) \int_{\partial_\ve I_j} \abs{\rho_j} =  e^{-\sigma} \abs{\partial_\ve \cG} \leq_a e^{-\sigma}\abs{\partial_{e^{-\lambda n} \ve}\cG}.
\end{equation}

Finally, adding \eqref{eq:bd_chopped_pairs}, \eqref{eq:bd_short_pairs} and \eqref{eq:bd_tail} together, we arrive at 
\begin{equation}
 \abs{\partial_\ve \cG_{n}} \leq_{4a+2D} 2(2^n+e^{-\sigma} e^{\lambda n})\abs{\partial_{e^{-\lambda n} \ve}\cG}  + 6\ve_0^{-1}\ve \abs{\cG}.
\end{equation}
\end{proof}

\begin{lem} \label{iterated_growth_lemma}
Suppose $\cG \in \cM_{a,b,B,\ve_0}$ with $a$ sufficiently large and $\ve_0$ sufficiently small. Then, there exist $C$, $\bar C$, and $0<\beta < \lambda$ such that,
\begin{equation} \label{eq:iterated_growth_lemma}
\abs{\partial_\ve \cG_{m}} \leq Ce^{\beta m} \abs{\partial_{e^{-\lambda m}\ve}\cG} + \bar C \ve, \text{ for all } m \in \bN \text{, and } \ve < \ve_0.
\end{equation}
\end{lem}

\begin{proof}
The result follows by choosing a fixed $n$ as in \Cref{fixed_param} and iterating \eqref{eq:growth_lemma} with this fixed $n$. Choose $n$, $\sigma$ large such that $2e^{4a+2D}(2^ne^{-\lambda n}+e^{-\sigma}) <1$ as in \Cref{fixed_param}. Let $\beta$ be such that $e^\beta=(2e^{4a+2D})^{1/n}(2^n+e^{-\sigma} e^{\lambda n})^{1/n}$. Then $e^{\beta n}e^{-\lambda n} <1$. That is, $\beta<\lambda$. For every $m \in \bN$, write $m=kn+r$, $0\leq r<n$.  Applying \eqref{eq:growth_lemma}, we have
\begin{equation} \label{eq:once_gl}
 \abs{\partial_\ve \cG_{n}} \leq 2e^{4a+2D}(2^r+e^{-\sigma}e^{\lambda r})\abs{\partial_{e^{-\lambda r} \ve}\cG_{kn}} + 6 \ve_0^{-1}\ve \abs{\cG_{kn}}.
\end{equation}
Let $C=2e^{4a+2D}(2^r+e^{-\sigma}e^{\lambda r})$. Applying \eqref{eq:growth_lemma} $k$ more times, we get
\begin{equation}
\begin{split}
\abs{\partial_\ve \cG_{m}} &\leq C e^{\beta kn}\abs{\partial_{e^{-\lambda r}e^{-\lambda kn} \ve}\cG} + 6\ve_0^{-1}\ve(e^{-\lambda r}/(1-e^{\beta-\lambda})+1) \abs{\cG} \\
&\leq C e^{\beta m}\abs{\partial_{e^{-\lambda m}\ve}\cG} + \bar C\ve \abs{\cG},
\end{split}
\end{equation}
where $\bar C = 6\ve_0^{-1}(e^{-\lambda r}/(1-e^{\beta-\lambda})+1)$. 
\end{proof}

\begin{rem}
The invariance of $\cM_{a,b,B,\ve_0}$ under iterations by $\sL_b^m$ follows by taking $a$, $m$, $B$ large, and $\ve_0$ small as in \Cref{fixed_param}. Note that $m$ and $B$ must be large enough that $Ce^{(\beta-\lambda)m}+\bar C/B < 1$ to guarantee $\abs{\partial_\ve \cG_m} \leq B \ve$ for all $\ve <\ve_0$.
\end{rem}

\section{Transversality} \label{trans}

Due to the neutrality of the $y$-direction in our setting, it is possible that measure stays on $x$-direction invariant curves that simply rotate in the $y$-direction. In this scenario the skew-product would not be mixing. To avoid such a scenario we need an assumption that forces measure to spread in different directions. We shall refer to this property as transversality. In our setting we assume that $\tau$ is not Lipschitz-cohomologous to a piecewise constant function on the joint partition of $\tau$ and $f$. In this section, we show that this condition implies a uniform non-integrability condition, which in turn implies the transversality notion that we use later to obtain a stretched-exponential decay of correlations. The material of this section is influenced by \cite{Tsu08}.

Note that \begin{equation}
DF_{(x,y)} = \begin{pmatrix}
  f '(x)& 0  \\
  \tau'(x) & 1  \\
 \end{pmatrix},
\end{equation}
which is independent of the second coordinate. Also, note that $DF$ preserves the cone $\sK_\eta=\{(u,v): \abs{v} \leq \eta\abs{u} \}$, where $\eta=\norm{\tau'/f'}_\infty/(1-\norm{1/f'}_\infty)$.

\begin{lem} \label{coho_tang_cones}
Suppose that for every $n \in \bN$, for every  $x \in \bT^1$, and inverse branches $h_1, h_2 \in \cH^{n}$,\footnote{We really mean every pair of inverse branches that have $x$ in their domain.}
\[
DF^{n}_{h_1(x)} \sK_\eta \cap DF^{n}_{h_2(x)}\sK_\eta \neq \{0\}.
\]
Then, $\tau$ is Lipschitz-cohomologous to a piecewise constant function on the joint partition of $f$ and $\tau$.
\end{lem}
\begin{proof}
The hypothesis implies that for every $x \in \bT^1$, $\cap_{h \in \cH^{n}} DF^{n}_{h(x)}\sK_\eta \neq \{0\}$. That is, the intersection contains a common direction $\theta(x,n)$. Moreover, since the cones $DF^{n}_{h(x)}\sK_\eta$ contract uniformly under iteration, the intersection $\cap_{n \in \bN}\cap_{h \in \cH^{n}} DF^{n}_{h(x)}\sK_\eta $ contains a unique direction $(1,\theta(x))$. Since $\theta(x)$ is invariant under $DF$, we have \begin{equation} \label{eq:pre_coho}
f' \cdot \theta \circ f = \tau' + \theta.
\end{equation}
Define $\phi(x) = \int_{0}^x \theta(t)dt$. Note that since $\theta(x)$ is bounded, $\phi$ is Lipschitz. Let $p$ be the left endpoint of a partition element of the joint partition of $f$ and $\tau$ and $x$ be a point in the same partition element. 
We may write $\tau(x)= \tau(p)+\int_p^x \tau'$, where $\tau(p)$ is interpreted as the value obtained by taking a one-sided limit. Substituting $\tau'=f' \cdot \theta \circ f - \theta$ from \eqref{eq:pre_coho} into this equation, and doing a change of variables yields,
\begin{equation}\label{eq:coho}
\tau(x) = \phi \circ f(x) - \phi(x) + (\tau(p) + \phi(p) - \phi\circ f (p)).
\end{equation}
It follows that $\tau$ is Lipschitz-cohomologous to a piecewise constant function on the joint partition of $f$ and $\tau$.
\end{proof}

\begin{lem} \label{trans_cones_UNI}
Suppose for $x \in \bT^1$, $n \in \bN$, and inverse branches $h_1, h_2 \in \cH^n$ holds
\[
DF^{n}_{h_1(x)} \sK_\eta \cap DF^{n}_{h_2(x)}\sK_\eta = \{0\}.
\]
Then, there exists $C_0:=C_0(n,x)$ such that 
\[
\abs{(\tau_n \circ h_1)'(x)-(\tau_n \circ h_2)'(x)} > C_0.
\]
\end{lem}
\begin{proof}
The hypothesis implies that the two cones $DF^{n}_{h_1(x)} \sK_\eta$ and $DF^{n}_{h_2(x)}\sK_\eta$ are a distance $C_0:=C_0(n,x)$ apart. Suppose $v_1, v_2 \in \bR$ satisfy $\abs{v_1}, \abs{v_2} \leq \eta$. Then $(1,v_1), (1,v_2) \in \sK_\eta$. Observe that for $j \in \{1,2\}$, $DF^{n}_{h_j(x)}(1,v_j)=((f^n)'\circ h_j(x), \tau' \circ h_j(x) + v_j)$. Therefore $(1, (\tau_n \circ h_j)'(x)+v_j h_j') \in DF^{n}_{h_j(x)} \sK_\eta$. Therefore, the two vectors are also  $C_0$ apart; that is,
\[
\abs{(\tau_n \circ h_1)'(x)+v_1 h_1'(x)- (\tau_n \circ h_2)'(x)-v_2 h_2'(x)} > C_0.
\]
Using the triangle inequality,
\[
\abs{(\tau_n \circ h_1)'(x)- (\tau_n \circ h_2)'(x))} > C_0 - \abs{v_1 h_1'(x)-v_2 h_2'(x)}.
\]
Taking $v_1=v_2=0$ implies the result.
\end{proof}

\begin{lem} \label{nbd_UNI}
Suppose $\tau$ is not Lipschitz-cohomologous to a piecewise constant function on the joint partition of $\tau$ and $f$. Then, there exists $x_0 \in \bT^1$, there exists $\ntrans \in \bN$, inverse branches $h_1, h_2 \in \cH^{\ntrans}$, a neighbourhood $V_{\ntrans}$ of $x_0$ contained in the open set $f^{\ntrans}(O_{h_1}) \cap f^\ntrans (O_{h_2})$, and a constant $C_1:=C_1(n_1,x_0)$ such that 
\begin{equation} \label{eq:nbd_UNI}
\abs{\left(\tau_{\ntrans} \circ h_1 - \tau_{\ntrans} \circ h_2\right)'(x)} > C_1 \text{ for every } x \in V_{\ntrans}. 
\end{equation}
\end{lem}

\begin{proof}
Suppose $\tau$ is not Lipschitz-cohomologous to a piecewise constant function on the joint partition of $\tau$ and $f$. \Cref{coho_tang_cones} implies that there exists $\ntrans \in \bN$, $x_0 \in \bT^1$, 
 and inverse branches $h_1, h_2 \in \cH^\ntrans$ such that 
\begin{equation} \label{eq:trans_cones}
DF^{\ntrans}_{h_1(x_0)} \sK_\eta \cap DF^{\ntrans}_{h_2(x_0)}\sK_\eta = \{0\}.
\end{equation}
\Cref{trans_cones_UNI} implies that there exists $C_0 = C_0(n_1, x_0)$ such that
\begin{equation} \label{eq:distant_cones}
\abs{\left(\tau_{\ntrans} \circ h_1 - \tau_{\ntrans} \circ h_2\right)'(x_0)} \geq C_0.
\end{equation}
By continuity of $(f^{\ntrans})'$ and $\tau_{\ntrans}'$ at $h_1(x_0)$ and $h_2(x_0)$ the cones $DF^{\ntrans}_{h_1(x_0)} \sK_\eta$ and $DF^{\ntrans}_{h_2(x_0)}\sK_\eta$ vary continuously in a neighbourhood of $x_0$ and so does the distance between them (i.e. $C(n,\cdot)$ varies continuously in a neighbourhood of $x_0$). It follows that there exists a neighbourhood $V_{\ntrans}$  of $x_0$ and a constant, which we again denote by $C_1:=C_1(\ntrans,x_0)$ such that
\begin{equation} \label{eq:nbd_distant_cones}
\abs{\left(\tau_{\ntrans} \circ h_1 - \tau_{\ntrans} \circ h_2\right)'(x)} \geq C_1 \text{ for every } x \in V_{\ntrans}.
\end{equation}
\end{proof}

\begin{cor}\label{persistent_nbd_UNI}
Suppose $\tau$ is not Lipschitz-cohomologous to a piecewise constant function on the joint partition of $\tau$ and $f$. Then, there exists $x_0 \in \bT^1$, there exists $\ntrans \in \bN$, inverse branches $h_1, h_2 \in \cH^{\ntrans}$, a constant $C_1:=C_1(n_1, x_0)$, and for every $n \geq \ntrans$ and every $l_1,l_2 \in \cH^{n-\ntrans}$ there exists a neighbourhood $V_n$ of $x_0$ contained in $f^n(O_{l_1 \circ h_1}) \cap f^n(O_{l_2 \circ h_2})$ such that 
\begin{equation}
\abs{\left(\tau_n \circ l_1 \circ h_1 - \tau_n \circ l_2 \circ h_2\right)'(x)} > C_1 \text{ for every } x \in V_{n}.
\end{equation}
\end{cor}
\begin{proof}
Let $x_0$, $\ntrans$, $V_{\ntrans}$ and $C_1$ be as in \Cref{nbd_UNI}. For every $n \geq \ntrans$ and $l_1,l_2 \in \cH^{n-\ntrans}$, by invariance of the cone, we have $DF^{n-\ntrans}_{l_j\circ h_j(x)}\sK_\eta \subset \sK_\eta$ for every $x \in f^n(O_{l_1 \circ h_1}) \cap f^n(O_{l_2 \circ h_2})$. If also $x \in V_{\ntrans}$ (the neighbourhood of $x_0$ from \Cref{nbd_UNI}), then $DF^{\ntrans}_{h_1(x)}DF^{n-\ntrans}_{l_1\circ h_1(x)}\sK_\eta \cap DF^{\ntrans}_{h_2(x)}DF^{n-\ntrans}_{l_2\circ h_2(x)}\sK_\eta = \{0\}$. That is, the cones $DF^{n}_{l_1\circ h_1(x)}\sK_\eta$ and $DF^{n}_{l_2\circ h_2(x)}\sK_\eta$ are at least distant $C_1$ apart. As in \Cref{trans_cones_UNI}, this transversality of the cones implies
\begin{equation}
\abs{\left(\tau_{\ntrans}\circ l_1 \circ h_1 - \tau_{\ntrans} \circ l_2 \circ h_2\right)'(x)} > C_1 \text{ for every } x \in V_n, 
\end{equation}
where $V_n:=V_\ntrans \cap f^n(O_{l_1 \circ h_1}) \cap f^n(O_{l_2 \circ h_2})$. 
\end{proof}

The following shows that any interval of positive length maps forward, while getting cut and expanded, in a way that at least two of its pieces overlap and simultaneously satisfy a condition similar to \eqref{eq:nbd_UNI}.
\begin{prop} \label{overlap_UNI}
Suppose $\tau$ is not Lipschitz cohomologous to a piecewise constant function on the joint partition of $\tau$ and $f$; and, in addition, $f$ is covering. There exists a constant $C_1$ such that for every interval $I$ with $0<\delta<\abs{I} \leq \ve_0 $,  there exists $n_\delta$ such that for every $n \geq n_\delta$, there exist $h_1,h_2 \in \cH^n$, such that $O_{h_1},O_{h_2} \subset I$ and $f^n(O_{h_1}) \cap f^n(O_{h_2})$ contains an interval $I_*$ of size $0<\ovl \leq \abs{I_*}$ on which holds\footnote{The quantity $\ovl$ depends on $\delta$, $n_\delta$ and the choice of $I$. Later we will get rid of the dependence on $I$ by a compactness argument.}
\[
\abs{\left(\tau_n \circ    h_1 - \tau_n \circ   h_2 \right)'} > C_1.
\]
\end{prop}

\begin{proof}
\Cref{persistent_nbd_UNI} implies that there exists $x_0$, $\ntrans$, inverse branches $\tilde h_1, \tilde h_2 \in \cH^{\ntrans}$; there exists a constant $C_1$;  and, for every $n \geq n_1$ and every $l_1,l_2 \in \cH^{n-n_1}$, there exists a neighbourhood $V_n$ of $x_0$ contained in $f^n(O_{l_1 \circ \tilde h_1}) \cap f^n(O_{l_2 \circ \tilde h_2})$, such that 
\begin{equation}
\abs{\left(\tau_n \circ l_1 \circ \tilde h_1 - \tau_n \circ l_2 \circ \tilde h_2\right)'(x)} > C_1 \text{ for every } x \in V_{n}.
\end{equation}

Since $f$ is covering, there exists $N(\delta)$  (recall that $\delta$ is the lower bound on the length of $I$) such that for every $n \geq N(\delta)+n_1=:n_\delta$ and every $l_1, l_2 \in \cH^{n-n_1}$ 
\[
l_j \circ \tilde h_j (V_{n}) \subset O_{l_j \circ h_j} \subset I \text{, for }j \in \{1,2\}.
\]
Note that the first inclusion is a consequence of the property that $V_n$, $n\geq n_1$, is contained in $f^n(O_{l_1 \circ \tilde h_1}) \cap f^n(O_{l_2 \circ \tilde h_2})$. Set $h_1 := l_1 \circ \tilde h_1$, $h_2 := l_2 \circ \tilde h_2$ and $I_* = V_{n}$. Then, we have
\[
\abs{\left(\tau_n \circ    h_1 - \tau_n \circ   h_2 \right)'} > C_1 \text{ for every } x \in I_*.
\]
Denote the length of $I_*$, the overlap interval, by $\ovl$. Note that $\ovl$ depends on $\delta$, $n$ and the choice of the initial interval $I$.
\end{proof}

\subsection{Transversality of standard pairs}
In this subsection we will state the transversality condition of \Cref{overlap_UNI} in terms of standard pairs. We will also get rid of the dependence of $\Delta$ on the choice of the interval $I$ using a compactness argument.

\begin{cond}[Transversality of standard pairs] \label{finite_transversality}
Consider a standard family $\cG$. For every $\delta>0$, there exists $n_\delta \in \bN$, a finite number $k:=k_\delta$ of pairs of inverse branches
\[
\{(h_{1,1},h_{1,2}) \dots (h_{k,1},h_{k,2})\} \subset  \cH^{n_\delta} \times \cH^{n_\delta} 
\]
such that for any standard pair $(I, \rho) \in \cG$, with $\abs{I} >3\delta$, the image standard family $\cG_{n_\delta}$ contains two standard pairs, obtained from the above finite collection of inverse branches, which overlap and are transversal on an interval of length no smaller than $\ovl=\ovl(k_\delta, n_\delta)$.  

More precisely, there exists $l \in \{1, \dots ,k_\delta\}$, $\ovl:=\ovl(k_\delta, n_\delta) >0$, standard pairs $(I_j, \rho_j)$, $j\in\{1,2\}$ such that there exists $U$, with $\ve_0/3 \leq \abs{U} \leq \ve_0$,\footnote{ $U$ is a choice of cutting and can be taken to be equal to $\bT^1$ if no cutting is necessary; that is, when $\abs{f^{n_\delta}(O_{h_{l,j}})}< \ve_0$.} such that
\[
I_j:=f^{n_\delta}(O_{h_{l,j}})\cap U, \rho_j := z_{ h_{l,j}}^{-1}e^{ib\tau_{n_\delta} \circ  h_{l,j}}\rho \circ h_{l,j} |( h_{l,j})'|, \text{ and }
O_{h_{l,j}}\subset I.
\]

Furthermore, $I_1 \cap I_2$ contains an interval $I_*$ of size $\ovl$ on which holds
\[
\abs{\left(\tau_{n_\delta} \circ h_{l,1}- \tau_{n_\delta} \circ  h_{l,2}\right)'} > C_1.
\]
Denote 
\begin{equation} \label{least_trans_domain}
M(n_\delta) := \min_{\stackrel{j\in\{1,2\}}{ l\in\{1, \cdots k_\delta\}}}\{\abs{O_{h_{l,j}}\cap h_{l,j}(U_j)}\}.
\end{equation}
\end{cond}

\begin{proof}
Divide the interval into subintervals of length $\delta$. Denote the finite collection of intervals by $\{J_l\}_{l=1}^{k_\delta}$. For each interval apply \Cref{overlap_UNI}. It follows that there exists $n:=n_\delta$ and finitely many inverse branches
\[
\{(h_{1,1},h_{1,2}) \dots (h_{k,1},h_{k,2})\} \subset  \cH^{n_\delta} \times \cH^{n_\delta} 
\]
such that
 $O_{h_{l,1}},O_{h_{l,2}} \subset J_l$ and $f^n(O_{h_{l,1}}) \cap f^n(O_{h_{l,2}})$ contains an interval of size $\Delta_l >0$\footnote{$\Delta_l$ depends on $\delta$ and $n_\delta$ in addition to $l$.} on which holds
\[
\abs{\left(\tau_n \circ    h_{l,1} - \tau_n \circ   h_{l,2} \right)'} > C_1.
\]
Let $\ovl = \min_{l\in \{1, \cdots, k_\delta \}} \Delta_l$. For any standard pair $(I,\rho)$ with $|I|> 3\delta$, $I$ contains at least one of the intervals $J_l$ of length $\delta$. As mentioned above, $J_l$ contains a pair of partition intervals $O_{h_{l,1}},O_{h_{l,2}}$ whose images overlap over an interval of length $\ve_0/3$ and are transversal. If these images are of length $< \ve_0$, by definition, they are the support of standard pairs:
\[
I_j:=f^n(O_{h_{l,j}}), \rho_j := z_{ h_{l,j}}^{-1}e^{ib\tau_n \circ  h_{l,j}}\rho \circ h_{l,j} |( h_{l,j})'|.
\]
However, if one of the images is of size greater than $\ve_0$, it must be shortened. In this case we may choose a cutting interval $U$, with  $\ve_0/3 \leq |U| \leq \ve_0$ that does not cut the overlap if $\ovl<\ve_0/3$. We also require that the cutting does not create other pieces of length $< \ve_0/3$. This can be done if $\ovl < \ve_0/3$ and if $\ovl \geq \ve_0/3$, we can consider a smaller overlap interval of size $< \ve_0/3$. With these considerations, we have obtained two standard pairs such that
\[
I_j:=f^n(O_{h_{l,j}})\cap U , \rho_j := z_{ h_{l,j}}^{-1}e^{ib\tau_n \circ  h_{l,j}}\rho \circ h_{l,j} |( h_{l,j})'|, 
\]
and such that $O_{h_{l,1}},O_{h_{l,2}} \subset I$ and $I_1 \cap I_2$ contains an interval of length $\ovl$ on which holds
\[
\abs{\left(\tau_n \circ    h_{l,1} - \tau_n \circ   h_{l,2} \right)'} > C_1.
\]

\end{proof}

\begin{rem}
The actual value of $3\delta$ for which the condition above is used, is determined in  \Cref{family_decay}. Also, note that for $n>n_\delta$, \Cref{finite_transversality} still holds but depends on $n$. So as long as we keep $n$ fixed, we may use \Cref{finite_transversality}, repeatedly. 
\end{rem}

\section{Weight reduction of standard families} \label{modi}

In this section our goal is to replace a standard family, after certain number of iterations, with an equivalent standard family of lower total weight. 

\begin{defin}[Equivalence]
Two standard families $\cG$ and $\tilde \cG$ are said to be \emph{equivalent} if $\rho_\cG = \rho_{\tilde\cG}$, i.e. if
\begin{equation}
\sum_{j \in \cJ} w_{\cG}(j) \rho_j = \sum_{k \in \tilde \cJ} w_{\tilde \cG}(k) \tilde \rho_k.
\end{equation}
\end{defin}

\begin{rem}
In this section we need to slightly increase the value of the parameter $a$. More precisely, we need $\left(e^{-\lambda n} + C_\tau/a + C_\kappa/a \right)\alpha_1^{-1} <1$.
This does not cause any problems since we could have chosen $a$ larger to begin with in \Cref{fixed_param}. In regards to $n$, we need $n > n_\delta$ and we choose it large enough that the above inequality holds and also $ae^{-\lambda n} < C_1/4$. Finally, we assume that $\abs{b} \geq 4\pi/(C_1 \ovl)$, where $C_1$ and $\ovl$ are related to transversality and were defined in the \Cref{trans}.
\end{rem}

\begin{lem} \label{singlepair_decay}
Suppose $\cG = \{(I, \rho)\} \in \cM_{a,b,B,\ve_0}$ is a singleton standard family with $\delta \leq \abs{I}$ for which \Cref{finite_transversality} holds. Then there exists there exists constants $\gamma>0$, $C$ such that for large $b$, letting $n_b=C\ln |b|$,  there exists a standard family $\cG_{n_b}^*$ equivalent to $\cG_{n_b}$  such that 
\[
\sum_{j \in \cJ_{n_b}^*} w_{\cG_{n_b}^*}(j) \leq  e^{-\gamma}w_\cG.
\]
\end{lem}
\begin{proof}
Let $(I_1,\rho_1),(I_2,\rho_2) \in \cG_{n}$ be the transversal standard pairs provided by \Cref{finite_transversality} applied to $\cG =\{(I, \rho)\}$.  Let $w_1=w_{\cG_{ n}}(h_1)$ and $w_2=w_{\cG_{ n}}(h_2)$ be the weights of these standard pairs.\footnote{We are using $\cH^{ n}$ for the index set of $\cG_{ n}$ to keep the notation simpler. To be precise, write $w_{\cG_{ n}}(j_n)$ with $j_{ n} \in \cJ_n$ as defined above.} Let $I_*$ be the interval of length $\ovl$ on which \Cref{finite_transversality} holds.  Let $\theta_1=\arg(\rho \circ   h_1)$ and $\theta_2=\arg(\rho \circ   h_2)$. Then, on the interval $I_*$, we may write $\abs{w_1\rho_1 + w_2\rho_2} = \abs{e^{i \Theta_b}w_1\abs{\rho_1} + w_2\abs{\rho_2}}$,
where 
\begin{equation}
\Theta_b= b(\tau_n \circ   h_1 - \tau_n \circ   h_2) - (\theta_1 - \theta_2).
\end{equation}

Our goal is to take out $\rho_1$ and $\rho_2$ from the family $\cG_n$ and replace them with other standard pairs (formed by combining parts of $\rho_1$ and $\rho_2$) and obtain a standard family $\cG_n^*$ which is still equivalent to $\cG_n$, but has a total weight strictly less than that of $\cG_n$. 

We will first show that the phase difference $\Theta_b$ grows at a certain rate.
\begin{claim}[Full phase oscillation]
For large $n$, there exists $C_1$ such that for $b \neq 0$, on the interval $I_*$, holds
\begin{equation} \label{eq:phase_bounds}
\frac{|b|C_1}{2} \leq \abs{\Theta_b'} \leq 2|b|(C_\tau +C_1) .
\end{equation}
 \end{claim}

\begin{proof}
Note that by \eqref{eq:arg_regularity}, on $I_*$,
\begin{equation} \label{eq:internal_phase}
\begin{split}
\abs{\theta_1'}&=\abs{\arg {(\rho \circ   h_1)}'} \leq a|b|h_1'| \leq a|b| e^{-\lambda n},\\
\abs{\theta_2'} &\leq a|b| e^{-\lambda n}. 
\end{split}
\end{equation}
Choose $n$ large enough\footnote{In addition to previous constraints.} that 
\begin{equation}
ae^{-\lambda n} < C_1/4.
\end{equation} 
Then,  $\abs{\theta_1 - \theta_2} < |b|C_1/2$. Also, \Cref{finite_transversality} implies $\abs{b(\tau_n \circ   h_1 - \tau_n \circ   h_2)'}> |b|C_1$ hence $\abs{\Theta_b'} > |b|C_1/2$. Finally, a simple estimate shows that  $\abs{\Theta_b'} \leq 2|b|(C_\tau +C_1)$. 
\end{proof}

Note that since $\Theta_b$ is $\sC^1$ and satisfies the bounds \eqref{eq:phase_bounds}, $\Theta_b'$ does not change sign in $I_*$.  Divide the range of $\Theta_b$ into intervals of length between $2\pi$ and $3\pi$, then the bounds on $\Theta_b'$ imply that $I_*$ will be divided into corresponding intervals $I_m$ of length $K_1 |b|^{-1}\leq \abs{I_m} \leq K_2 |b|^{-1}$, where $K_1:=\pi/(C_\tau+C_1)$ and $K_2:=6\pi/C_1$. To clarify, these are intervals on which $\Theta_b$ makes one full oscillation, but less than one and a half full oscillations. Of course we must make sure $I_*$ is large enough to fit at least one such interval $I_m$. That is we need $K_1\abs{b}^{-1} \leq \abs{I_*}$. This can be accomplished by choosing $b$ large enough:\footnote{This is the only restriction on $b$.} 
\begin{equation}\label{eq:b_0}
\abs{b} \geq K_1/\abs{I_*}=4\pi/(C_1\ovl):=b_0.
\end{equation}

We like to combine some part of $\rho_1$ and $\rho_2$ to take advantage of their cancellation. Since the modulus of these standard pairs are not smooth, if we combine them blindly, we might lose the $\sC^1$-smoothness required for the argument of a standard pair. For this reason we do the following splitting of the standard pairs into good parts, with a constant modulus, which we can combine; and bad parts, which we do not combine in this round. 

For a function $\rho\in \L^1(I,\bC)$ with $\int_I\abs{\rho} \neq 0$, denote $N(\rho)=\rho/\int_I\abs{\rho}$. For $j\in \{1,2\}$, we split $(I_j,\rho_j)$ into two standard pairs $
\left(I_j,N(\tilde\rho_j) \right) $, $\left(I_j,N(\bar \rho_j) \right)$, such that:
\begin{equation}\label{eq:split}
\bar \rho_j = c e^{i \Theta_j},
\tilde\rho_j = (\abs{\rho_j}-c)e^{i\Theta_j}, \text{ where } c= \frac{e^{-a}}{2}, \Theta_j = b(\tau_n \circ h_j)+ \theta_j.
\end{equation}
Associate to them the weights $\bar w_j = w_j \int_{I_j}\abs{\bar \rho_j}$, $\tilde w_j = w_j \int_{I_j}\abs{\tilde \rho_j}$.
\begin{claim}[After splitting]
For $j \in \{1,2\}$,
\begin{eqnarray}
\rho_j &=& \bar \rho_j + \tilde \rho_j\\
w_j &=& \bar w_j + \tilde w_j \\
w_j \rho_j &=&  \bar w_j N(\bar \rho_j)+ \tilde w_j N(\tilde \rho_j)\\
H(\bar \rho_j) &\leq& a \\
H(\tilde \rho_j) &\leq& 4a  \label{eq:remainder_reg}\\
\abs{\Theta_j'} &\leq& a|b| \left(e^{-\lambda n} + \frac{C_\tau}{a}\right) \label{eq:new_phase}
\end{eqnarray}
\end{claim}

\begin{proof}
The first four statements are easy to prove. To prove  \eqref{eq:remainder_reg}, note that $c =(1/2)e^{-a} \leq (1/2)\inf \abs{\rho_j}$.
Therefore,
\begin{equation}\label{eq:trick}
\frac{\abs{\abs{\rho_j(x)}-c}}{\abs{\abs{\rho_j(y)}-c}}\leq  \frac{\abs{\abs{\rho_j(x)}-\abs{\rho_j(y)}}+ \abs{\rho_j(y)} - c}{\abs{\rho_j(y)}-c} = 1 +  \frac{\abs{\abs{\rho_j(x)}-\abs{\rho_j(y)}}}{\abs{\rho_j(y)}-c}
\end{equation}
But, $\abs{\rho_j(y)} - (1/2) \abs{\rho_j(y)} = (1/2)\abs{\rho_j(y)}  \geq \inf\abs{\rho_j} \geq c$. Hence, $
\abs{\rho_j(y)} - c \geq \frac{1}{2}\abs{\rho_j(y)}$,
and we have:
\begin{equation}
\begin{split}
\frac{\abs{\abs{\rho_j(x)}-c}}{\abs{\abs{\rho_j(y)}-c}} &\leq 1 + 2\frac{\abs{\abs{\rho_j(x)}-\abs{\rho_j(y)}}}{\abs{\rho_j(y)}} \leq 1+ 2\abs{\frac{\abs{\rho_j(x)}}{\abs{\rho_j(y)}}- 1}\\
&\leq 1 + 2 \abs{e^{a\abs{x-y}^\alpha} - 1} = 1+ 2 (2a|x-y|^\alpha)\\
&\leq  e^{4a\abs{x-y}^\alpha}.
\end{split}
\end{equation}
The inequality \eqref{eq:new_phase} follows from \eqref{eq:internal_phase} and \eqref{eq:roof_regularity}.
\end{proof}
We now describe how to combine $\bar\rho_1$ and $\bar\rho_2$. Note that the modulus of these functions is constant and equal to $c$. We need the following result.
\begin{claim}[$J_m$. Preparing for a controlled cancellation of $\bar \rho_1$ and $\bar \rho_2$] \label{cases}
Suppose $w_2 \leq w_1$.\footnote{Otherwise, interchange indices and do the same proof.} For every $\alpha_1 \in (0,1/2]$ and $\alpha_2 \in [(\sqrt{7}-1)/2,1)$, for every $m$, there exists a subinterval $J_m \subset I_m$ with $K_3|b|^{-1}\leq \abs{J_m} \leq K_4\abs{b}^{-1}$ such that for every $\kappa_0 \geq w_2/(2w_1)$
\begin{equation} \label{eq:cancelation}
\alpha_1 c\left(\kappa_0 w_1+w_2 \right) \leq \abs{\kappa_0 w_1\bar \rho_1+w_2 \bar \rho_2} \leq c(\kappa_0  w_1 + \alpha_2 w_2).
\end{equation}
\end{claim}

\begin{proof}
Choose $K_3, K_4$ such that $1/4 \leq \cos(\Theta_b) \leq 1/2$  on $J_m$. This can be done because the phase difference $\Theta_b$ makes a full oscillation in $I_m$. 
The left side of \eqref{eq:cancelation} is easy to prove and does not require a restriction on $\kappa_0$. Let us prove the right side.

Note that, on one hand, using $\cos(\Theta_b) \leq 1/2$,
\[
\begin{split}
\abs{\kappa_0 w_1\bar \rho_1+w_2 \bar \rho_2}^2 &= \kappa_0^2w_1^2\bar\rho_1^2+w_2^2\bar\rho_2^2+ 2\kappa_0 w_1w_2\abs{\bar\rho_1}\abs{\bar \rho_2} \cos(\Theta_b)\\
&=c^2\left(\kappa_0^2w_1^2+w_2^2+ 2\kappa_0 w_1w_2 \cos(\Theta_b)\right) \\
&\leq c^2\left(\kappa_0^2w_1^2+w_2^2+ \kappa_0 w_1w_2 \right).
\end{split}
\]
On the other hand,
\[
\left(c( \kappa_0  w_1 + \alpha_2 w_2)\right)^2 = c^2\left(\kappa_0^2w_1^2+ \alpha_2^2 w_2^2 + 2 \alpha_2\kappa_0 w_1 w_2\right).
\]
Hence it suffices to show
\[
0 \leq  w_2 (\alpha_2^2 - 1) + \kappa_0 w_1(2\alpha_2 - 1).
\]
Solving for $\alpha_2$, it suffices to show
\begin{equation}\label{eq:alpha_2}
\alpha_2 \geq \sqrt{1+\frac{\kappa_0 w_1}{ w_2} + \left(\frac{\kappa_0 w_1}{w_2}\right)^2}- \frac{\kappa_0 w_1}{w_2}.
\end{equation}
If $w_2 / (2w_1) \leq \kappa_0 $, then  $ (\kappa_0 w_1)/w_2 \geq 1/2$. From the graph of $x\mapsto \sqrt{1+x+x^2}-x$, for $x \in [0,\infty)$, it follows that the right hand side of \eqref{eq:alpha_2} is at most $(\sqrt{7}-1)/2$. Therefore \eqref{eq:alpha_2} is satisfied for any $\alpha_2 \geq (\sqrt{7}-1)/2$.
\end{proof}

Fix $\kappa_0=w_2/(2w_1)$. Choose a smooth function $\kappa \in \sC^1(I_*, [1-\kappa_0,1])$ such that $\kappa =1-\kappa_0$ on the middle third of $J_m$, denoted $J_m'$, and  $\kappa = 1$ outside $J_m$. Note that, taking the length of a connected component of $J_m\setminus J_m'$ into account, $\kappa$ can be chosen such that $\abs{\kappa'}<C_\kappa \kappa_0 |b|$ on $I_m$.  Define\footnote{Note that we are assuming $w_2 \leq w_1$.} 
\begin{equation}\label{eq:new_pairs}
\bar\rho_{1*}:= \kappa \bar \rho_{1} \text{, and } \bar\rho_{2*} := \bar \rho_{2} + (1-\kappa)(w_1/w_2)\bar \rho_{1}.
\end{equation}
The domain of the definition above is the overlap interval $I_*$; however, for $j\in \{1,2\}$, we may extend the domain of $\bar\rho_{j*}$ to the interval $I_j$ so that $\bar\rho_{j*} = \bar\rho_j$. This should be clear from the definition of $\kappa$ and \eqref{eq:new_pairs}. We intend to replace $\bar \rho_1, \bar\rho_2$ with $\bar \rho_{1*}, \bar \rho_{2*}$. We will not touch $\tilde \rho_1, \tilde \rho_2$ except to normalize them.

Define a new family
\begin{equation}
\begin{split}
\cG_{n}^* &:=\left(\cG_{ n} \setminus \{(I_1,\rho_1), (I_2,\rho_2)\} \right) \cup \\ & \left\{\left(I_1,N(\bar\rho_{1*})\right), \left(I_2,N(\bar\rho_{2*})\right),\left(I_1,N(\tilde\rho_{1})\right), \left(I_2N(\tilde\rho_{2}\right)\right\}.
\end{split}
\end{equation}
with associated weight measure $w_{\cG_n^*}$ that is the same as $w_{\cG_n}$ except for the modified standard pairs. For the modified standard pairs, define the new weights by $\bar w_{1*} := w_1\int_{I_1}\abs{\bar \rho_{1*}}$, $\bar w_{2*} :=  w_2 \int_{I_2}\abs{\bar \rho_{2*}}$, $\tilde w_{1} :=  w_1\int_{I_1}\abs{\tilde \rho_{1}}$, $\tilde w_{2} := w_2 \int_{I_2}\abs{\tilde\rho_{2}}$.

Now we check that the new collection $\cG_n^*$ is a standard family equivalent to $\cG_n$.
\begin{claim}[After cancellation]
We have:
\begin{eqnarray}
w_1\rho_1 + w_2 \rho_2 &=& w_1 \bar \rho_{1*} + w_1 \tilde \rho_1 + w_2 \bar \rho_{2*} + w_2 \tilde \rho_2. \label{eq:equival}\\
H(\bar \rho_{j*}) &\leq& a |b|, \forall j \in \{1,2\} \label{eq:mod_reg_new}\\
\abs{\Theta_{j*}'} &\leq& a|b|, \forall j \in \{1,2\} \label{eq:arg_reg_new}\\
\abs{\partial_\ve \cG_n^*} &\leq& C_* \abs{\partial_\ve\cG_n} \label{eq:bd_of_mod}.
\end{eqnarray}
\end{claim}

\begin{proof}
The equality \eqref{eq:equival} follows from definition. Indeed, \eqref{eq:new_pairs} implies
\[
w_1\bar \rho_{1*} + w_2 \bar \rho_{2*}=w_1\bar \rho_1 + w_2 \bar \rho_2,
\]
which in turn implies  \eqref{eq:equival}. Hence $\cG_n^*$ and $\cG_n$ are equivalent.

To prove  \eqref{eq:mod_reg_new}, note that by construction $\bar\rho_{1*}, \bar\rho_{2*}$ are $\sC^1$. Hence it suffices to show $\abs{\bar\rho_{j*}'} \leq a|b| \abs{\bar\rho_{j*}}$. 

For $\bar\rho_{1*}$ this condition is easier to check than for $\bar\rho_{2*}$. Let us check, for $\bar\rho_{2*}$ the stronger condition: $\abs{\bar\rho_{2*}'} \leq a|b| \abs{\bar\rho_{2*}}$ for $a$ and $n$ large enough.\footnote{The previous choices of $a$ and $n$ need to be updated.} Outside $J_m$, $\bar \rho_{2*}$  satisfies this condition because $\bar\rho_{2}$ does. On $J_m$, differentiating $\bar\rho_{2*}$ and using \eqref{eq:split} and \eqref{eq:new_phase} yield, 
\[
\begin{split}
\abs{\bar\rho_{2*}'} &\leq c \left( \abs{\Theta_2'} + \abs{1- \kappa}\frac{w_1}{w_2} \abs{\Theta_1'}\right) + \frac{w_1}{w_2}\abs{\kappa'}\abs{\rho_{1}} \\
&\leq a|b|\left( e^{-\lambda n} + \frac{C_\tau}{a}\right) c\left(1 + \abs{1- \kappa}\frac{w_1}{w_2} \right) + c  \frac{w_1}{w_2}\abs{\kappa'}.
\end{split}
\]
Now observe that since $K_3|b|^{-1}\leq \abs{J_m} \leq K_4|b|^{-1}$  and  $1-w_2/(2w_1) \leq \kappa \leq 1$, we have $\abs{\kappa'}< C_\kappa w_2/(2w_1)|b|$. Also, $\abs{1-\kappa} \leq w_2/(2w_1)$. Therefore,
\[
\frac{w_1}{w_2}\abs{\kappa'} \leq C_\kappa \left(1+\abs{1-\kappa}\frac{w_1}{w_2}\right)|b|
\]
Therefore,
\[
\begin{split}
\abs{\bar\rho_{2*}'} &\leq a|b|\left( e^{-\lambda n} + \frac{C_\tau}{a} + \frac{C_\kappa}{a}\right)c\left(1 + \abs{1- \kappa}\frac{w_1}{w_2} \right) \\
&\leq a|b|\left( e^{-\lambda n} + \frac{C_\tau}{a} + \frac{C_\kappa}{a}\right) \frac{\abs{\bar\rho_{2*}}}{\alpha_1},
\end{split}
\]
where the last inequality follows from the left hand side of \eqref{eq:cancelation}. Recall that the left hand side of \eqref{eq:cancelation} requires no restriction on $\kappa_0$. It simply follows from the phase difference satisfying $\cos(\Theta_b)=\cos (\Theta_1-\Theta_2) \geq 1/4$ and $\alpha_1 < 1/2$.
Take $a,n$ large to conclude.

The inequality \eqref{eq:arg_reg_new} is a consequence of $\abs{\bar \rho_{j*}'} \leq  a|b| \abs{\bar\rho_{j*}}$ since $\abs{\Theta_{j*}'} \leq \abs{\bar \rho_{j*}'}/\abs{\bar\rho_{j*}}$.

To prove \eqref{eq:bd_of_mod}, note that
\[
\abs{\partial_\ve \cG_n^*} \leq \abs{\partial_\ve\cG_n} + \int_{\partial_\ve I_1}w_1 \abs{\tilde \rho_1}  + \int_{\partial_\ve I_2}w_2 \abs{\tilde \rho_2}  + \int_{\partial_\ve I_1}w_1\abs{\bar\rho_{1*}}   + \int_{\partial_\ve I_2}w_2\abs{\bar\rho_{2*}}.
\]
The first three terms are simply bounded by $3\abs{\partial_\ve \cG_n}$. The last two terms are bounded by $\abs{\partial_\ve \cG_n}+ \int_{\partial_\ve I_*} (w_1 \abs{\bar \rho_{1*}} + w_2 \abs{\bar \rho_{2*}})$. Note that  using \eqref{eq:new_pairs}, $w_1 \abs{\bar \rho_{1*}} + w_2 \abs{\bar \rho_{2*}} \leq w_1\abs{\bar \rho_1}+ w_2\abs{\bar \rho_1}$. Putting all this together and noting that $\bar\rho_1=\bar\rho_2=c$, we have
\[ 
\begin{split} 
\abs{\partial_\ve \cG_n^*} &\leq 3\abs{\partial_\ve\cG_n} + \abs{\partial_\ve\cG_n}+ \int_{\partial_\ve I_*}(w_2\abs{\bar\rho_{2}} +w_1 \abs{\bar \rho_1} ) \\
&\leq 4\abs{\partial_\ve\cG_n} + c \ve (w_1+w_2)\\
&\leq 5c \abs{\partial_\ve\cG_n}.
\end{split}
\]
\end{proof}

Let us check that the total weight of the new standard family is less than the old one. 
\begin{claim}
$
\sum_{j \in \cJ_{n_b}'} w_{\cG_{n_b}'}(j) \leq  e^{-\gamma}w_\cG.
$
\end{claim}
\begin{proof}
First, note that on  $J_m'$, $\kappa = 1- \kappa_0$. Therefore on $J_m'$ (and using the right hand side of \eqref{eq:cancelation}) we have
\[
\begin{split}
w_1\abs{\bar \rho_{1*}} + w_2\abs{\bar \rho_{2*}} &\leq w_1 (1-\kappa_0) \abs{\bar \rho_{1}}+ w_2\abs{\bar \rho_{2} + \kappa_0\frac{w_1}{w_2}\bar \rho_{1}}\\
&\leq w_1 (1-\kappa_0) \abs{\bar \rho_{1}}+ \alpha_2 w_2\abs{\bar \rho_{2}} + w_1\kappa_0\abs{\bar \rho_{1}} \\
&= w_1 \abs{\bar \rho_1} + \alpha_2 w_2 \abs{\bar\rho_2}.
\end{split}
\] 
Outside $J_m'$, by definition, $w_1\abs{\bar \rho_{1*}}+w_2\abs{\bar \rho_{2*}} \leq w_1\abs{\bar \rho_{1}}+w_2\abs{\bar \rho_{2}}$.

Then, 
\[
\begin{split}
\int _{I_*} \left(w_1\abs{\bar \rho_{1*}} + w_2\abs{\bar \rho_{2*}}  \right)  &\leq \sum_{m } \int_{I_m \setminus J_m'} \left(w_1 \abs{\bar \rho_1} + w_2 \abs{\bar \rho_2}\right)  + \int_{J_m'} \left( w_1\abs{\bar \rho_{1}} + \alpha_2 w_2\abs{\bar \rho_{2}} \right)  \\
&=  \int_{I_*}w_1 \abs{\bar \rho_1}  + \sum_{m }\int_{I_m \setminus J_m'}w_2 \abs{\bar \rho_2} + \int_{J_m'} \alpha_2 w_2\abs{\bar \rho_{2}}\\
&=  \int_{I_*}w_1 \abs{\bar \rho_1}  + \sum_{m }\int_{I_m}w_2 \abs{\bar \rho_2} -(1-\alpha_2) \int_{J_m'} w_2\abs{\bar \rho_{2}} . 
\end{split}
\]
Since $K_1|b|^{-1}\leq \abs{I_m} \leq K_2|b|^{-1}$ and $(1/3)K_3|b|^{-1}\leq \abs{J_m'} \leq (1/3)K_4|b|^{-1}$, and $\abs{\bar \rho_2}$ is a constant, \Cref{Fed} implies that $\int_{J_m'} w_2\abs{\bar \rho_{2}} \geq (\abs{J_m'}/\abs{I_m}) \int_{I_m}w_2 \abs{\bar \rho_2} \geq K_3/(3K_2) \int_{I_m}w_2 \abs{\bar \rho_2}$. Therefore, 
\[
\int _{I_*} \left(w_1\abs{\bar \rho_{1*}} + w_2\abs{\bar \rho_{2*}}  \right)  \leq \int_{I_*}w_1 \abs{\bar \rho_1}  + \sum_{m }\int_{I_m}w_2 \abs{\bar \rho_2} \left(1-(1-\alpha_2)K_3/(3K_2) \right)
\]
Set $\alpha_2':= 1-(1-\alpha_2)K_3/(3K_2)$. Note that $0< \alpha_2'<1$ and it does not depend on $m$. Pulling it out of the integral, we have
\[
\int _{I_*} \left(w_1\abs{\bar \rho_{1*}} + w_2\abs{\bar \rho_{2*}}  \right)  \leq  \int_{I_* } w_1 \abs{\bar\rho_1}  +  \alpha_{2}'\int_{I_*}w_2 \abs{\bar \rho_2}  . 
\]

Add $\int_{I_1 \setminus I_*}w_1\abs{\bar \rho_{1*}}  + \int_{I_2 \setminus I_*}w_2\abs{\bar \rho_{2*}}  =  \int_{I_1\setminus I_*}w_1\abs{\bar\rho_{1}}  + \int_{I_2 \setminus I_*}w_2\abs{\bar \rho_2} $ to the above inequality and apply the same argument (using \Cref{Fed}) to conclude that there exists $0<\alpha_2''<1$ such that
\[
\bar w_{1*} + \bar w_{2*} = w_1\int _{I_1} \abs{\bar \rho_{1*}}  + w_2\int_{I_2}\abs{\bar \rho_{2*}}  \leq   w_1\int_{I_1}\abs{\bar \rho_1} + \alpha_2'' w_2\int_{I_2}\abs{\bar \rho_2}  = \bar w_1 + \alpha_2'' \bar w_2. 
\]
 
Observe that since $\alpha_2''<1$ and the rest of the standard pairs in the family were not modified, the total weight of the new standard family is less than the original one. Estimating the total weight more precisely, for large $n$, 
\[
\begin{split}
|\cG_n^*|  &=
 |\cG_n|-(w_1+w_2) + \bar w_{1*} + \tilde w_1  +\bar w_{2*} + \tilde w_2 \\
&\leq |\cG_n|-(w_1+w_2)  + \bar w_{1} + \tilde w_1 + \alpha_2'' \bar w_{2} + \tilde w_2 \\
&\leq |\cG_n| -  (1- \alpha_2'')\bar w_2  \\
&\leq w_{\cG} -  (1- \alpha_2'')CM(n) w_\cG  \\
&\leq  e^{-\gamma}  w_{\cG}. 
\end{split}
\]
Recall that $w_{\cG}$ is the weight of the original standard family consisting of a single standard pair. In the next to last inequality we used the lower bound on $w_2$.  Indeed, by definition of $\bar w_2$, \Cref{iteration}, change of variables, and finally the definition of $M(n)$ (see \Cref{finite_transversality}), we have
\[
\begin{split}
\bar w_2 = w_2 \int_{I_2} |\bar \rho_2| =  c |I_2|w_2 &\geq c|I_2|w \int_{I_2} \abs{\rho} \circ h_2 \abs{h_2'} \\
&\geq c \ovl w \inf_I \abs{\rho} \abs{(I \cap O_{h_2}) \cap h_2(U_\ell)}\\
&\geq c \ovl w\inf_I \abs{\rho}M(n) := CM(n) w_\cG.
\end{split}
\]
\end{proof}

There is one last issue to resolve. The members of $\cG_{n}^*$ may not satisfy $H(\rho) \leq a $. In order to achieve this, we simply iterate the family for a time $C\ln|b|$. Indeed, suppose $(I, \rho) \in \cG_{n}^*$. Note that  $H(\rho) \leq a |b|$. Following the proof of property \eqref{eq:mod_reg_1} in \Cref{invariance}, note that after $\tilde n$ more iterations, every image pair $\tilde \rho \in \cG_{n+\tilde n}^*$ satisfies:
\[
H(\tilde \rho) \leq a \left(|b| e^{-\lambda \tilde n} + \frac{D}{a}\right).
\]
Now, it is clear that, for $b \geq b_0$ if 
$\tilde n > \lambda^{-1}\ln\left( |b|(1-Da^{-1})^{-1}\right)=:\tilde n_b$,
then $H(\tilde\rho) \leq a$. Finally, note that if $C$ is chosen large enough, then $n_b:=C\ln|b|$ dominates $n+\tilde n_b$. We have shown the existence of the standard family $\cG_{n_b}^*$ as claimed in \Cref{singlepair_decay}. \end{proof}

\begin{prop} \label{family_decay}
There exists $0< \gamma_1 < \gamma $ and for $|b| \geq b_0$ there exists $n_b=C \ln|b|$, such that for any standard probability family $\cG \in \cM_{a,b,B,\ve_0}$, there exists a standard family $\cG_{n_b}^* \in \cM_{a,b,B,\ve_0}$ equivalent to $\cG_{n_b}$ such that $\abs{\cG_{n_b}^*} \leq e^{-\gamma_1}$.
\end{prop}
\begin{proof}
Choose $\delta>0$ small enough that $B \delta < 1/4$. Then for large $m$, $\abs{\partial_\delta \cG_m} \leq B\delta <1/4$. It follows that the total weight of the standard pairs of $\cG_{m}$ that have width larger than $\delta$ is $L\geq 3/4$.  Let $S$ denote the total weight of standard pairs that have width $\leq \delta$. Note that $\abs{\cG}=\abs{\cG_m}=S+L = 1$.  Fix $\tilde n_b$ as in \Cref{singlepair_decay}. Using \Cref{singlepair_decay}, 
\[
\abs{\cG_{m+\tilde n_b}^*} \leq  S + e^{-\gamma} L \leq (S+L)-(1-e^{-\gamma})L \leq 1-(1-e^{-\gamma})3/4 =: e^{-\gamma_1}. 
\]
Take $n_b = m +\tilde n_b$.
\end{proof}

\section{Proofs of the main propositions} \label{dens_decay}
\begin{lem} \label{Lb_decay} There exists $C>0$,  $\gamma_2>0$, such that if $\{(I,\rho)\} \in \cM_{a,b,B,\ve_0}$ with $a,b,B$ sufficiently large and $\ve_0$ sufficiently small, then for every $n \in \bN$, 
\begin{equation}
\norm{  \sL_b^n \rho}_{\L^1  } \leq C e^{-\frac{\gamma_2}{\ln{|b|}} n}.
\end{equation}
\end{lem}

\begin{proof}
The result follows by repeatedly applying \Cref{family_decay} and renormalizing the total weight of the standard family at every step. Indeed, let $n_b=C\ln|b|$ as in \Cref{family_decay}.  After $k+1$ repetitions we have
\[
\begin{split}
\norm{  \sL_b^{(k+1)n_b} \rho}_{\L^1  } &\leq  \sum_{j_{k+1}\in \cJ_{k+1}} w_{\cG_{(k+1)n_b}^*}(j_{k+1}) \\
&\leq \sum_{j_k \in \cJ_k} e^{-\gamma_1} w_{\cG_{kn_b}^*}(j_k) \\
&\leq e^{-(k+1)\gamma_1} 
\text{.}
\end{split}
\]
This means, for every $m \in \bN$ ($m=kn_b+r_b$, $0 \leq r_b<n_b$),
\[
\begin{split}
\norm{ \sL_b^{m} \rho}_{\L^1  } &\leq \norm{ \sL_b^{r_b} \sL_b^{kn_b}\rho}_{\L^1  } \leq \norm{ \sL_0^{r_b}\abs{ \sL_b^{kn_b}\rho}}_{\L^1  } = \norm{ \sL_b^{kn_b}\rho}_{\L^1  } \\
&\leq e^{k\gamma_1} \leq e^{-(\frac{m}{n_b}-1) \gamma_1}\leq e^{\gamma_1}e^{-\frac{\gamma_1}{n_b}m} \leq C_{\gamma_1}e^{-\frac{\gamma_1}{n_b}m}.
\end{split}
\]
Note that here by $\sL_b^0\rho$, we mean $\rho$. Set $\gamma_2=\gamma_1/C$, where $n_b = C\ln|b|$, to conclude.  
\end{proof}

\begin{proof}[Proof of \Cref{main_estimate}] \label{prf_main_estimate}
We will show that there exists a constant $C$, such that for every $b \geq b_0$, for every $n \in \bN$, 
\begin{equation}
\norm{ \sL_b^n}_{\mixnorm} \leq Ce^{-\frac{\gamma_2}{\ln{|b|}} n}.
\end{equation}

Write $g=g - c + c$, where $c = 1+ |g|_\alpha/a+\sup|g|$ ($\abs{g}_\alpha$ being the $\alpha$-H\"older constant of $g$). Note that both $|c|$ and  $\abs{g-c}=1+|g|_\alpha/a+\sup{\abs{g}}-g$ are bounded below by $1$. Of course, $H(c)=0$. Calculating as in \eqref{eq:trick}, we get 
\[
\frac{\abs{g(x)-c}}{\abs{g(y)-c}} \leq \frac{\abs{g(x)-g(y)}}{\abs{g(y)-c}}+1 \leq \frac{\abs{g}_\alpha \abs{x-y}^\alpha}{\abs{g(y)-c}} +1 \leq e^\frac{\abs{g}_\alpha \abs{x-y}^\alpha}{\abs{g(y)-c}}.
\] By the choice of $c$, it follows that $\abs{g}_\alpha/(\abs{g(y)-c}) \leq  a$. Therefore, $H(g-c) \leq a$. 

Let $g_c=g-c$. Partition the domains of $g_c$ and $c$ into intervals of length $\ve_0>0$ (small enough as in \Cref{fixed_param}) and renormalize the restricted functions. Then $g_c=\sum_{j=1}^{L}  w_j \rho_j$ and $c=\sum_{k=1}^{L} \tilde w_k \tilde \rho_k$, where $L=\lceil\ve_0^{-1}\rceil$, $\{\rho_j\},\{\tilde\rho_k\}$ are standard families with parameters $a,b,B,\ve_0$ as defined before; and $\{w_j\}, \{\tilde w_k\}$ are their associated weights. Then $g = \sum w_j \rho_j + \sum \tilde w_k \tilde \rho_k$, where $\sum w_j+ \sum \tilde w_k = \norm{g_c}_{\L^1} + \norm{c}_{\L^1}$. Apply \Cref{Lb_decay} and obtain for $b \geq b_0$,  for every $n \in \bN$,
\[
\begin{split}
\norm{ \sL_b^n g}_{\L^1} &\leq \sum_{j=1}^L w_j \norm{ \sL_b^n \rho_j}_{\L^1} + \sum_{j=1}^L \tilde w_k \norm{ \sL_b^n \tilde \rho_k}_{\L^1} \\
&\leq C_{\gamma_1}e^{-\frac{\gamma_2}{\ln{|b|}} n}\left(\sum_{j=1}^L  w_j+ \sum_{j=1}^L  \tilde w_k \right) \\
&\leq C_{\gamma_1}e^{-\frac{\gamma_2}{\ln{|b|}} n}\left(\norm{g_c}_{\L^1}+\norm{c}_{\L^1}\right) \\
& \leq CC_{\gamma_1} e^{-\frac{\gamma_2}{\ln{|b|}} n} \norm{g}_{\sC^\alpha}
\end{split}
\]
\end{proof}

\begin{proof}[Proof of \Cref{small_b}] \label{prf_small_b}
It suffices to show that $\sL_b$ has spectral radius $e^{-r}<1$ when $b \neq 0$. By the assumptions on $\B$, we know that $\sL_b$ has essential spectral radius $e^{-r'}<1$ and spectral radius at most $1$. This means that there are only finitely many eigenvalues outside the disk of radius  $e^{-r'}$. Therefore, if we show that there are no eigenvalues on the unit circle, it follows that the spectral radius is strictly less than $1$. In turn, this implies the existence of $C$ such that $\norm{\sL^n_b}_{\B} < Ce^{-rn}$.

Let us show that there are no eigenvalues of $\sL_b$ on the unit circle for $b \neq 0$. Fix $b \geq b_0$ and suppose that there exists $g$ and $\lambda \in \bC$ satisfying $\abs{\lambda}=1$ such that
 $\sL_b g = \lambda g$. Since $\abs{\sL_b g} \leq \sL_0 \abs g$,  it follows that $\abs{g} \leq \sL_0 \abs{g}$. Since, $\int \abs{g} = \int \sL_0 \abs{g}$,  this implies that 
 $\abs{g} = \sL_0 \abs{g}$. Since $\abs{\lambda}=1$, this means that $
\abs{ \sL_{b} g  }  = \sL_0 \abs{g}$.

Using the definition of $\sL_b$  observe that, for every  $y$, the arguments of the complex numbers $g(x) e^{ib\tau(x)}$ must be equal for all $x$ such that $f(x)=y$. Choose some $k\in \bN$ such that $b k > b_0$, where $b_0$ is as in \eqref{eq:b_0}. The arguments of the complex numbers $g^k(x) e^{ibk\tau(x)}$ are equal for all  $f(x)=y$. This means that $ \abs{ \sL_{kb} g^k  } = \sL_0 \abs{g^k}$. It also means that  $ \abs{ \sL_{kb}^n g^k  } = \sL_0^n \abs{g^k}$ for any $n\in \bN$ (we could have considered the $n$-th power from the start of the argument).
 We have that 
 \[
  \int \abs{ \sL_{kb}^n g^k  } dm= \int  \sL_0^n \abs{g^k} dm= \int \abs{g^k} dm\quad \quad \text{for all $n\in \bN$}.
 \]
Using the estimate for large $|b|$, if $g \in \B$, then the left hand side vanishes as $n\to \infty$ whereas the right hand side is fixed and non-zero. 
\end{proof}


\bibliography{../Bibliography/Skew}

\end{document}